\documentclass[a4paper, 12pt, draft, reqno]{amsart}
\usepackage[final]{graphicx}
\usepackage{amsmath}
\usepackage{amssymb}
\usepackage{latexsym}
\usepackage[cmtip,matrix,arrow]{xy}
\usepackage{rotating}
\UseComputerModernTips
\CompileMatrices

\newcommand{\ZZ}{{\mathbb Z}}

\newcommand{\RR}{{\mathbb R}}
\newcommand{\NN}{{\mathbb N}}
\newcommand{\CC}{{\mathbb C}}

\newcommand{\HH}{{\mathbb H}}

\newcommand{\calB}{{\mathcal B}}

\DeclareMathOperator{\Hom}{Hom}

\DeclareMathOperator{\soc}{soc}
\DeclareMathOperator{\rad}{rad}

\DeclareMathOperator{\res}{res}

\DeclareMathOperator{\supp}{supp}

\DeclareMathOperator{\Mod}{-mod}

\DeclareMathOperator{\ind}{ind}

\begin{document}
\theoremstyle{plain}
\newtheorem{thm}{Theorem}[subsection]
\newtheorem{prop}[thm]{Proposition}
\newtheorem{lem}[thm]{Lemma}
\newtheorem{cor}[thm]{Corollary}
\newtheorem{conj}[thm]{Conjecture}
\newtheorem{qn}[thm]{Question}
\newtheorem{claim}[thm]{Claim}
\newtheorem{defn}[thm]{Definition}
\theoremstyle{definition}
\newtheorem{rem}[thm]{Remark}
\newtheorem{ass}[thm]{Assumption}
\newtheorem{example}[thm]{Example}

\setlength{\parskip}{1ex}

\title[On Brauer algebra simple modules over the complex field]{On Brauer algebra simple modules over the complex field} 

\author{Maud De Visscher} 
\email{M.Devisscher@city.ac.uk}
\author{Paul P Martin}
\email{p.p.martin@leeds.ac.uk}
\maketitle

\begin{abstract}
This paper gives two results on the simple modules for the Brauer algebra over the complex field. First we describe the module structure of the restriction of all simple modules. Second we give a new geometrical interpretation of Ram and Wenzl's construction of bases for \lq $\delta$-permissible' simple modules.
\end{abstract}

\section{Introduction}
\newcommand{\kk}{{\mathtt k}} 

\subsection{} Classical Schur-Weyl duality relates the representations
of the general linear group and the symmetric group via commuting
actions on tensor space. The Brauer algebra was introduced by Brauer
in 1937 to play the role of the symmetric group when one replaces the
general linear group by the orthogonal or symplectic group. 
For any non-negative integer $n$,
any commutative ring $\kk$, and any $\delta\in \kk$, we can define the
Brauer algebra $B_n(\delta)$ as being the $\kk$-algebra 
with basis 
all pair partitions of $2n$. We can represent these basis elements as diagrams (so-called Brauer diagrams) having $2n$ vertices arranged in 2 rows of $n$ vertices each, such that each vertex is linked to precisely one other vertex. The multiplication is then given by concatenation, removing all closed loops, and scalar multiplication by $\delta^k$ where $k$ is the number of closed loops removed.
It's easy to see that $B_n(\delta)$ is generated by the set $\{\sigma_i, e_i \, : \, 1\leq i \leq n-1\}$ where $\sigma_i$ and $e_i$ are given in Figure \ref{sigmae}.

\begin{figure}[ht]
\includegraphics[width=10cm]{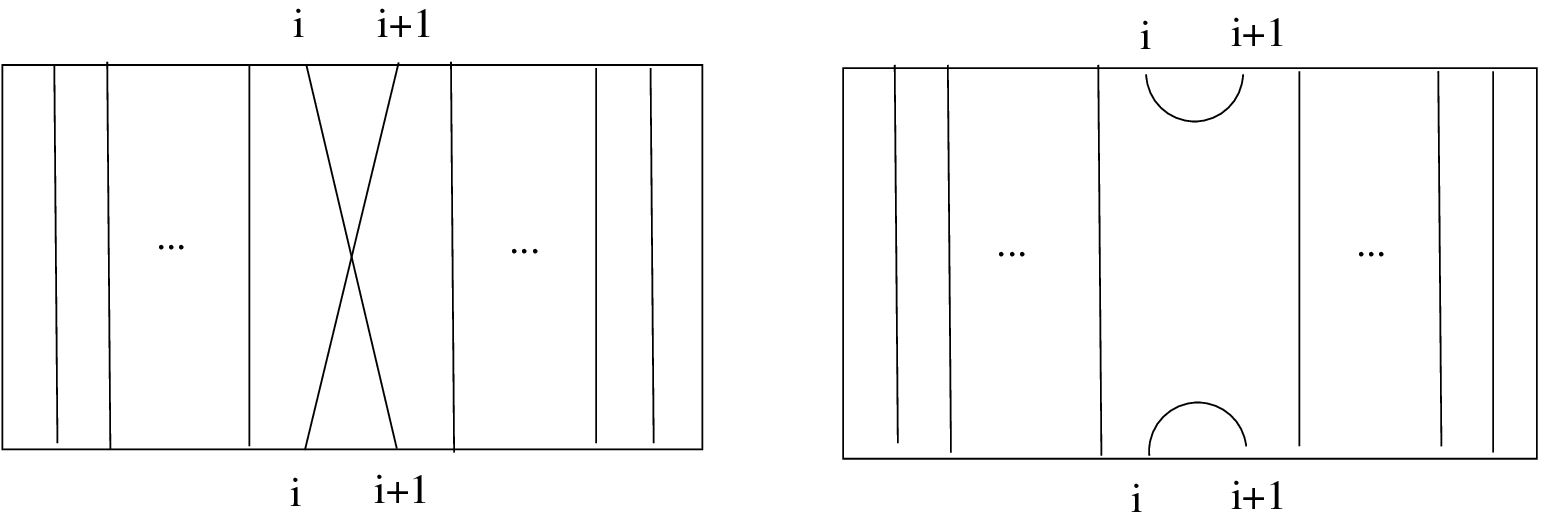}
\caption{}
\label{sigmae}
\end{figure}

The symmetric group algebra $\kk \Sigma_n$ appears naturally as the subalgebra of $B_n(\delta)$ generated by the $\sigma_i$'s. Note that $\kk\Sigma_n$ also occurs as a quotient of $B_n(\delta)$ as explained below. This turns out to be very helpful in studying the representation theory of $B_n(\delta)$.

Assume 
for a moment 
that 
$\delta$ is a unit.
Consider the 
idempotent given by $e=\frac{1}{\delta} e_{n-1}$. Then it is easy to see that 

\begin{equation}\label{eBe}
eB_n(\delta)e\cong B_{n-2}(\delta) \quad \mbox{and}\quad B_n(\delta)/B_n(\delta)eB_n(\delta) \cong \kk\Sigma_n. 
\end{equation}

\noindent Now 
fix $\kk=\CC$ and 
recall that the simple $\CC\Sigma_n$-modules are indexed by partitions of $n$, that is for each partition $\lambda$ we have a (simple) Specht module $S^\lambda$. Using (\ref{eBe}) we can easily deduce by induction on $n$ that the simple modules for $B_n(\delta)$ are indexed by the set $\Lambda_n$ of partitions of $n, n-2, n-4, \ldots$. For each $\lambda\in \Lambda_n$, we denote the corresponding simple module by $L_n(\lambda)$.

When $B_n(\delta)$ is semisimple, the simple modules can be
constructed explicitely by \lq inflating' (or \lq globalising') the
corresponding Specht module, see for example 
\cite{HP}. However the algebra $B_n(\delta)$ is not always semisimple. In 1988, Wenzl showed in \cite{W} that if $B_n(\delta)$ is  not semisimple  then $\delta \in \ZZ$, and in 2005, Rui gave an explicit criterion for semisimplicity in \cite{Rui}.

In this paper, we study the simple modules when $B_n(\delta)$ is not
semisimple. So we will assume that $\delta\in \ZZ$. 
For the moment 
we will also assume that $\delta \neq 0$. 
In this case, $B_n(\delta)$ is a quasi-hereditary algebra with respect
to the opposite order to the one given by the size of partitions. (In
fact, we will work with a refinement of this order, see Section
2.2). In particular, the indecomposable projective modules
$P_n(\lambda)$ ($\lambda \in \Lambda_n$) have a filtration by standard
modules $\Delta_n(\lambda)$ ($\lambda\in \Lambda_n$). The standard
modules can be constructed 
explicitly 
(as inflation of Specht modules, as in the semisimple case) and we have surjective homomorphisms
$$P_n(\lambda)\twoheadrightarrow \Delta_n(\lambda) \twoheadrightarrow L_n(\lambda)$$
for each $\lambda\in \Lambda_n$. Now the decomposition matrix
$D_{\lambda \mu}=[\Delta_n(\mu):L_n(\lambda)]$ has been determined by
the second author in  \cite{M} and its inverse is given in
\cite{CD}. This gives a closed form for the dimension of the simple
modules (although the coefficients of $(D_{\lambda \mu})^{-1}$ are not
easy to compute in practice). 

\subsection{} We have natural embeddings of the Brauer algebras
$$B_{n-1}(\delta)\hookrightarrow B_n(\delta)\hookrightarrow B_{n+1}(\delta)$$
by adding a vertical edge between the last vertex in each row of every Brauer diagram.
So we have corresponding restriction functors $\res_n : B_n(\delta)\Mod \rightarrow B_{n-1}(\delta)\Mod$ and induction functors $\ind_n : B_n(\delta)\Mod \rightarrow B_{n+1}(\delta)\Mod$. For partitions $\lambda$ and $\mu$, we write $\lambda \triangleright \mu$ (resp. $\lambda \triangleleft \mu$) if $\lambda$ is obtained from $\mu$ by adding (resp. removing) a box to its Young diagram. From \cite{DWH} we have exact sequences
\begin{equation}\label{resDelta}
0\rightarrow \oplus_{\mu\triangleleft \lambda} \Delta_{n-1}(\mu)\rightarrow \res_n \Delta_n(\lambda) \rightarrow \oplus_{\mu \triangleright \lambda}\Delta_{n-1}(\mu) \rightarrow 0, \quad \mbox{and}
\end{equation}
\begin{equation}\label{indDelta}
0\rightarrow \oplus_{\mu\triangleleft \lambda} \Delta_{n+1}(\mu)\rightarrow \ind_n \Delta_n(\lambda) \rightarrow \oplus_{\mu \triangleright \lambda}\Delta_{n+1}(\mu) \rightarrow 0
\end{equation}
where we define $\Delta_{n-1}(\mu)=0$ when $\mu\notin \Lambda_{n-1}$. 

The first objective of this paper is to describe the corresponding result for all simple modules. More precisely, we describe completely the module structure of $\res_nL_n(\lambda)$ for all $\lambda\in \Lambda_n$ and all non-negative integers $n$.

\subsection{} Walk bases for standard modules for generic values of $\delta$ were given by Leduc and Ram in \cite{LR}. Their construction relies on complex combinatorial objects such as the King polynomials (first introduced in \cite{EK}). These bases do not specialise to $\delta\in \mathbb{Z}$ (except in very low rank). However, it follows implicitly from \cite{RamWenzl92} that the truncation of these representations to certain \lq $\delta$- permissible up-down tableaux' gives bases for the \lq $\delta$-permissible' simple modules.

More recently \cite{CDM2} introduced a geometric characterisation of the representation theory of the Brauer algebra. It turns out that the combinatorics used in \cite{RamWenzl92} and \cite{LR} can be explained in a uniform and natural way in this geometrical context. In particular, we obtain a striking characterisation of the roots of the King polynomials.

Motivated by this, the second objective of this paper is to recast the
contruction of \cite{LR} in the geometrical setting. 
This provides a unification of the classical and modern approaches,
but is also done 
with a view to treating arbitrary simple modules (and other characteristics) in further work.

\subsection{Structure of the paper.} In Section 2, we recall and extend the necessary setup from \cite{CDM2} for the geometrical interpretation of the representation theory of the Brauer algebra $B_n(\delta)$. In Section 3, we recall the construction of weight diagrams and cap diagrams associated to every partition $\lambda$ and integer $\delta$ introduced in \cite{M} and \cite{CD}. We develop some of their properties and recall how these can be used to describe the blocks and the decomposition numbers for $B_n(\delta)$. In Section 4 we give a complete description of the module structure of the restriction from $B_n(\delta)$ to $B_{n-1}(\delta)$ of every simple module in terms of cap diagrams. We start Section 5 by recalling the representations constructed by Leduc and Ram for the generic Brauer algebra. We then give a geometric interpretation of the combinatorics used in their construction and deduce, by specialisation and truncation, explicit bases for an important class of simple modules.

\section{Geometrical setting}

\subsection{Euclidian space and reflection groups.} 

Consider the  
space $\RR^\NN$ consisting of all
(possibly infinite) $\RR$-linear combination of the symbols
$\epsilon_i$ ($i\in \NN$). For each $x=\sum_{i\in \NN}x_i \epsilon_i$,
write $x=(x_1, x_2, x_3, \ldots)$. 
The inner product on 
finitary elements in 
$\RR^\NN$ is given by $\langle \epsilon_i, \epsilon_j\rangle =
\delta_{ij}$. 
Now define $W$ to be the infinite reflection group on $\RR^\NN$ of type $D$ generated by  the reflections $(i,j)_\pm$ ($i<j\in \NN$)  where
$$(i,j)_\pm : (\ldots , x_i , \ldots , x_j, \ldots) \mapsto (\ldots ,
\pm x_j , \ldots , \pm x_i , \ldots ).$$
Define $W_+$ to be the subgroup generated by  $(i,j)_+$ ($i<j\in \NN$). So $W_+$ is the infinite reflection group on $\mathbb{R}^\mathbb{N}$ of type $A$.
The group $W$ (resp. $W_+$) defines a set $\HH$ (resp. $\HH_+$) of hyperplanes corresponding to
the reflections $(i,j)_\pm$ (resp. $(i,j)_+$) on $\RR^\NN$. 
We define the \emph{degree of singularity} of an element $x\in \mathbb{R}^\mathbb{N}$, denoted by ${\rm deg}(x)$, to be the number of hyperplanes in $\HH$ containing $x$, that is  the number of pairs of entries $x_i, x_j$ ($i< j$)
satisfying $x_i=\pm x_j$. 
The set of hyperplanes $\HH$ (resp. $\HH_+$)
subdivide $\RR^\NN$ into so-called $W$-alcoves, (resp. $W_+$-alcoves), see \cite{Humphreys90}. 
Define the element $\rho\in \mathbb{R}^\mathbb{N}$ by
$$\rho=(0,-1,-2,-3,\ldots ).$$
Now define the \emph{dominant chamber} $X_+$ to be the $W_+$-alcove containing $\rho$, and the \emph{fundamental alcove} to be the $W$-alcove containing $\rho$.

\subsection{Embedding of the Young graph.}

Recall that the Young graph $\mathcal{Y}$ has vertex set the set $\Lambda=\cup_n \Lambda_n$ of all partitions and has an edge between two partitions $\lambda$ and $\mu$ if $\lambda \triangleright \mu$ or $\lambda \triangleleft \mu$.

\begin{prop}\label{deltawalk} Let $\lambda\in \Lambda_n$ and $\delta\in \ZZ$. The dimension of $\Delta_n(\lambda)$ is given by the number of (undirected) walks of length $n$ starting at $\emptyset$ and ending at $\lambda$.
\end{prop}

\begin{proof}
This follows from (\ref{resDelta}) by induction on $n$. 
\end{proof}

For each $\delta\in \ZZ$, we will now define an embedding of the graph $\mathcal{Y}$ into $\mathbb{R}^\mathbb{N}$. This embedding is the key to all the geometrical tolls for Brauer algebra representation theory.

\newcommand{\ZZZ}{{\mathcal Z}}

Define $\ZZZ$ as the graph with vertex set $\RR^\NN$ and an edge
$(x,x')$ whenever $x-x' = \pm \epsilon_i$ for some $i$. 
For $x \in \RR^\NN$ define $\ZZZ(x)$ as the connected component of
$\ZZZ$ containing $x$. 
Define $\ZZZ_+$ as the subgraph of $\ZZZ$ on vertices in the dominant
chamber $X_+$. Define $\ZZZ_+(x)$ as the connected component of $\ZZZ_+$ containing $x$. 
A walk on $\ZZZ_+$ is called a dominant walk. 

For each partition $\lambda = (\lambda_1, \lambda_2, \lambda_3, \ldots )$ (where $\lambda_i=0$ for all $i>>0$), consider the transpose partition $\lambda^T=(\lambda^T_1, \lambda^T_2, \lambda^T_3, \ldots)$. For each $\delta\in \ZZ$, define $\rho_\delta \in \RR^\NN$ by
$$
\rho_\delta = (-\frac{\delta}{2}, -\frac{\delta}{2}-1,
-\frac{\delta}{2}-2, -\frac{\delta}{2}-3, \ldots )
= -\frac{\delta}{2}(1,1,1,...) +\rho.$$
Now define the embedding $e_\delta : \mathcal{Y} \rightarrow \ZZZ$ by setting for each vertex $\lambda\in \Lambda$,
\begin{equation}\label{embedding}
e_\delta(\lambda)=\lambda^T +\rho_\delta.
\end{equation}
Note that $e_\delta(\lambda)\in X_+$ for all $\lambda\in \Lambda$ and all $\delta\in \ZZ$. In fact we have the following important observation.

\begin{lem}\label{embeddingiso} For every $\delta\in \ZZ$ the map $e_\delta : \mathcal{Y} \rightarrow \ZZZ_+(\rho_\delta)$ is a graph isomorphism.
\end{lem}

Using Lemma \ref{embeddingiso} we can rephrase Proposition \ref{deltawalk} as follows.

\begin{prop}
Fix $\delta\in  \ZZ$.
Points $x\in \RR^\NN$ reachable by undirected dominant walks on $\ZZZ$
of length $n$
from $\rho_\delta$ index the standard modules of $B_n(\delta)$. Moreover the number of undirected dominant walks on $\ZZZ$ from $\rho_\delta$ to $x$ gives the dimension of the corresponding standard module.
\end{prop}

\subsection{Representations of Temperley-Lieb algebras.}

To explain our geometrical programme in this paper we mention an analogous
situation in Lie theory --- specifically the 
representation theory of the Temperley--Lieb algebra $TL_n(\delta)$
(this is, via Schur--Weyl duality, 
simply the $sl_2$ case of a wider $sl_N$ phenomenon). We refer the reader to \cite[Section 12]{MM} and references therein for more details.

For $sl_2$ one should replace $\RR^\NN$ with $\RR$, replace
$\ZZZ$ with the corresponding graph (whose connected components are
simply chains of vertices)
and $W$ and $W_+$
with the reflections groups of type affine-$A_1$ and $A_1$
respectively, acting on $\RR$. In Figure \ref{fig:TLrep-walk}, we see
different sets of walks on a connected component of $\ZZZ$. The
hyperplanes or walls in $\HH$ are denoted by solid thick lines in
these pictures. 
There is, in principle, a representation for each choice of position
of the $A_1$-wall.
The relative position of the first affine wall depends on $\delta$ and
on the ground field. 
Figure \ref{fig:TLrep-walk}(a), 
shows all walks from the origin to a given point, which form a basis
for a module (isomorphic to a 
Young module in this case) when the $A_1$-wall is in generic position. 
Figure \ref{fig:TLrep-walk}(b) shows the basis of dominant walks
for a Temperley--Lieb Specht module, obtained when the $A_1$-wall is
in the `natural' position.
Figure \ref{fig:TLrep-walk}(c) then shows the subset of walks restricted to regular points, which we shall call \emph{restricted walks}
(in this example there is only a single such walk),  giving a basis for 
the simple head of the Specht module for a suitable $\delta$.
(Indeed
bases for arbitrary Temperley--Lieb  simples can be described using 
a refinement of the same
technology.)

\begin{figure}
\includegraphics[width=1.82in]{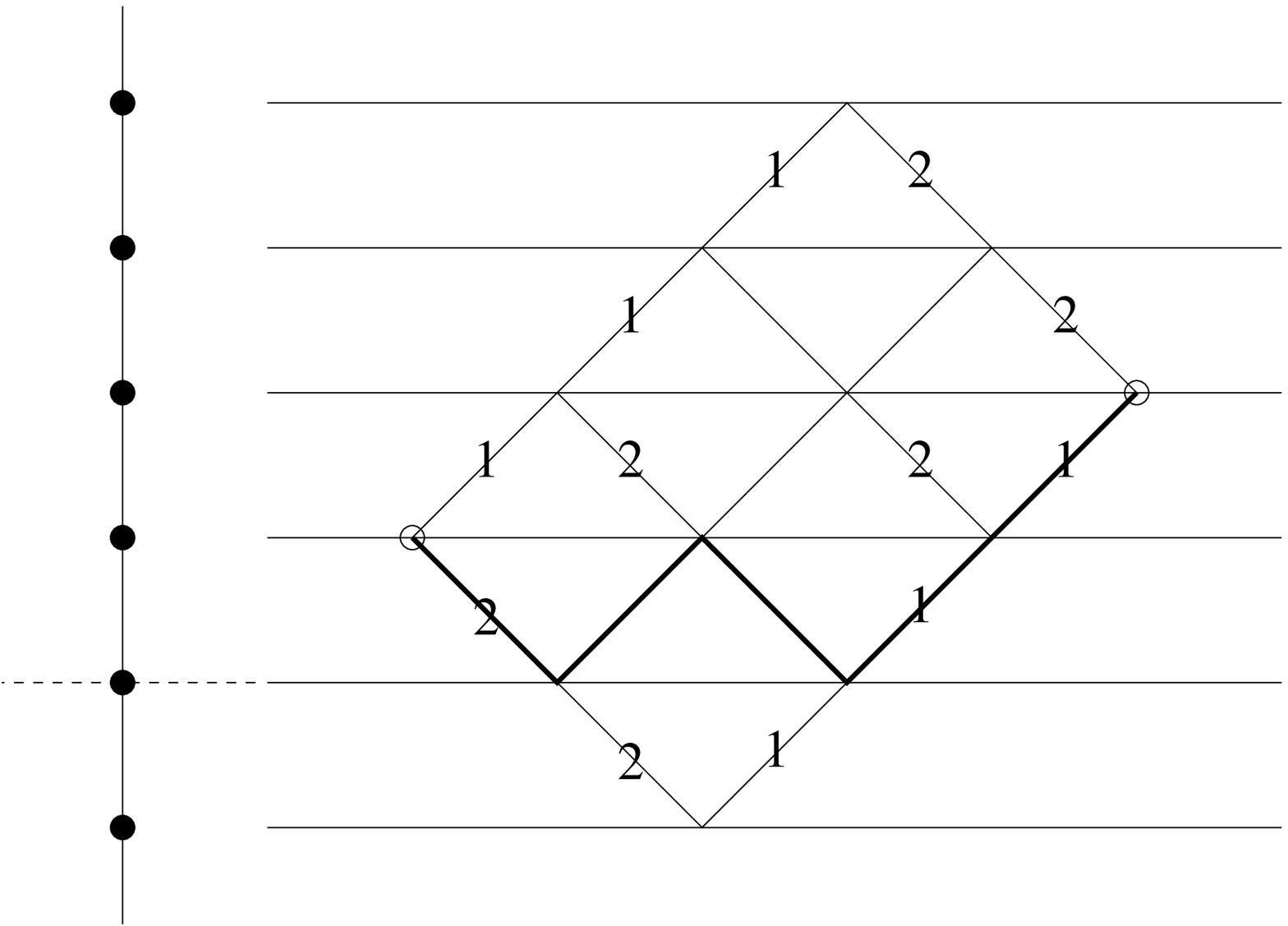}
\qquad
\includegraphics[width=1.82in]{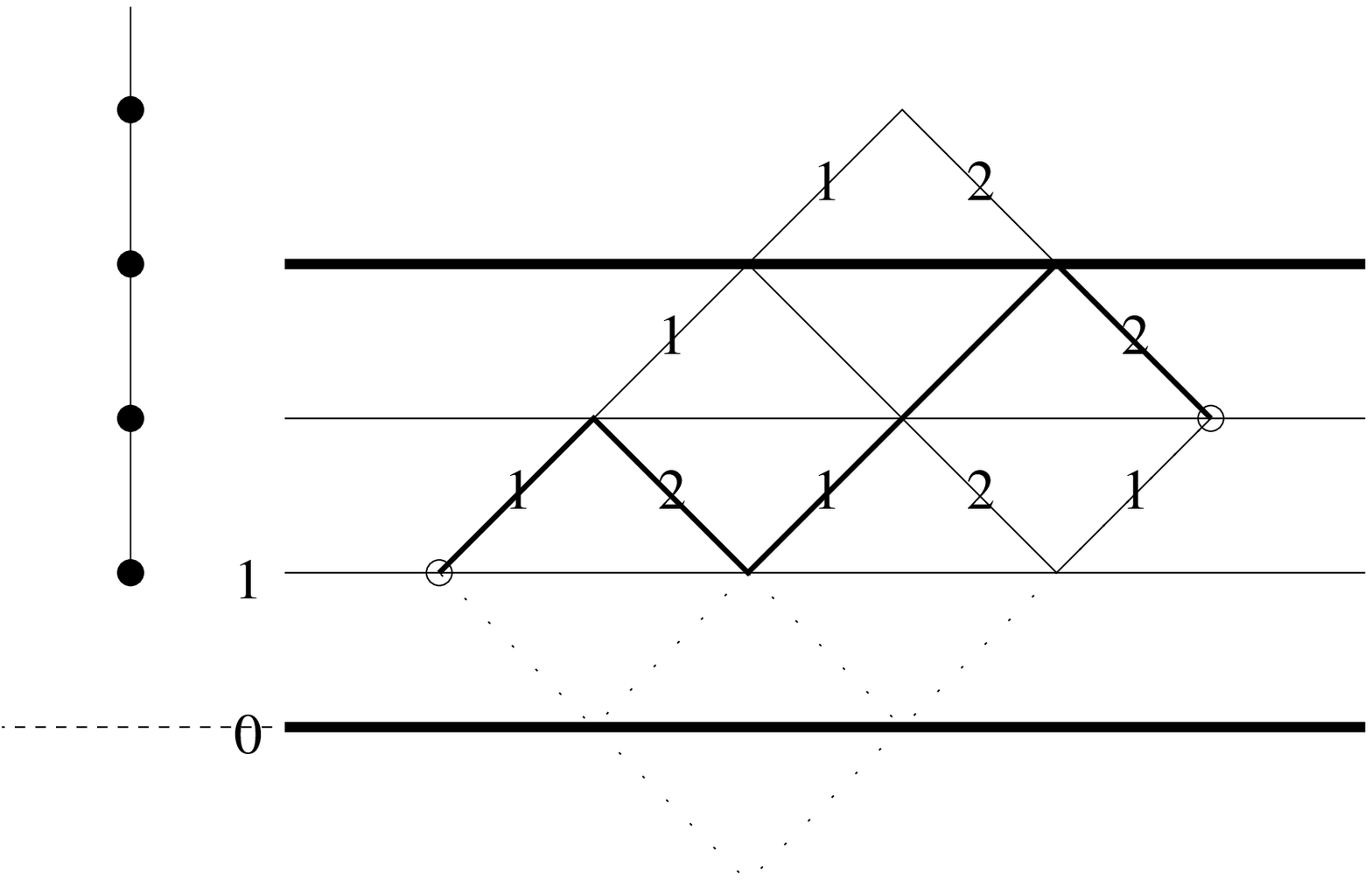}
\qquad
\includegraphics[width=1.82in]{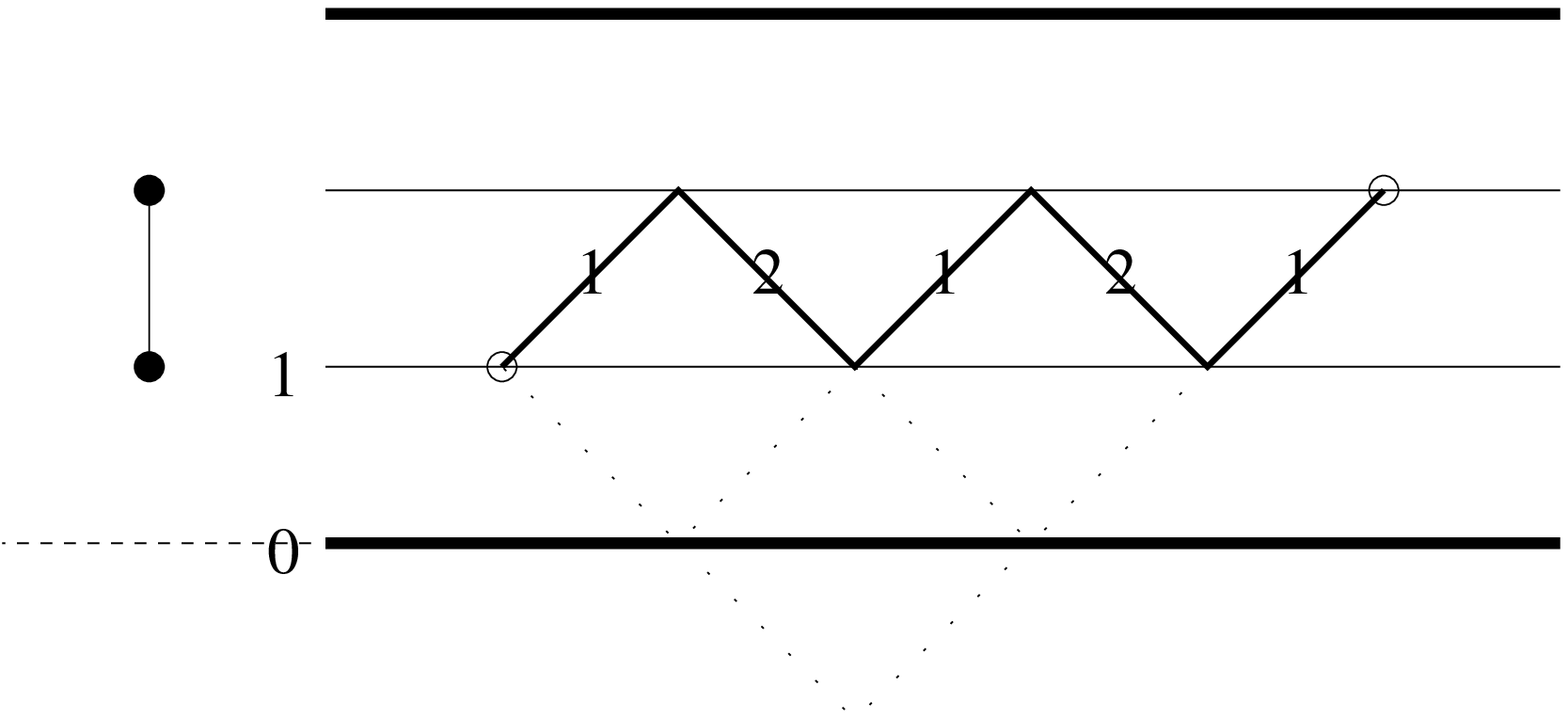}
\caption{(a), (b), (c) 
Walks on the $\ZZZ$-graph in type-$A_1$.
\label{fig:TLrep-walk}}
\end{figure}

Furthermore, the only off-diagonal entries in the
`unitary' representations of $TL_n(\delta)$ generators corresponding to
these walk bases 
are between (two) walks differing at a single point.
The difference is a reflection of the point 
in a certain hyperplane in $\RR$. The mixing
depends on the `height' of the hyperplane (i.e. the vertical axis in
Figure \ref{fig:TLrep-walk}), and vanishes at certain heights corresponding to affine walls 
(this is the `quantisation' of Young's famous hook-length orthogonal
form, expressed geometrically)
--- this explains why, in this case,  restricted walks decouple
from the rest. 


This example is nothing more than an analogy for us, since there is no corresponding piece of Lie
theory underlying our case. Nonetheless we will see in Section 5 that the features described above with the $W_+$-wall in the natural position also hold here. But first we need to describe the appropriate analogue of \lq restricted walks'.

\subsection{$\delta$-regularity and the $\delta$-restricted walks.}

By analogy with the Temperley-Lieb case, we might consider walks restricted to regular points. Note however that the degree of singularity of $\rho_\delta$ is given by
$${\rm deg}(\rho_\delta)=\left\{ \begin{array}{ll}
0 & \mbox{if $\delta
\geq 0$} \\
-m & \mbox{if $\delta = 2m<0$ or $2m+1<0$.}
\end{array}\right.$$

So we will first rescale the notion of regularity in a homogeneous way.
 
\begin{defn}\label{regsing} (i) For $x\in \mathbb{R}^\mathbb{N}$, we say that $x$ is $\delta$-regular if ${\rm deg}(x)={\rm deg}(\rho_\delta)$, and that $x$ is $\delta$-singular if ${\rm deg}(x)>{\rm deg}(\rho_\delta)$.\\
(ii) For a partition $\lambda\in \Lambda$ we define the $\delta$-degree of singularity of $\lambda$, denoted by ${\rm deg}_\delta(\lambda)$, by
$${\rm deg}_\delta(\lambda)={\rm deg}(e_\delta(\lambda)).$$
(iii) We say that $\lambda$ is $\delta$-regular (resp. $\delta$-singular) if $e_\delta(\lambda)$ is $\delta$-regular (resp. $\delta$-singular).
\end{defn}

Now we can define the restricted region in a homogeneous way as follows.

\begin{defn}\label{restrictedgraph}
(i) Define the $\delta$-restricted graph $\ZZZ_\delta$ to be the maximal connected subgraph of $\ZZZ_+$ containing $\rho_\delta$ such that all vertices are $\delta$-regular.\\
(ii) We define $\mathcal{Y}_\delta$ to be the inverse image of $\ZZZ_\delta$ under the map $e_\delta$, and we set $A_\delta$ to be the vertex set of $\mathcal{Y}_\delta$.\\
(iii) We call walks on $\ZZZ_\delta$, or $\mathcal{Y}_\delta$, $\delta$-restricted walks.
\end{defn}

\begin{rem} (i) Note that the set $A_\delta$ corresponds precisely to the set of $\delta$-permissible partitions defined in \cite{W}.\\
(ii) Note also that for $\delta 
\geq 0$ the set $A_\delta$ corresponds precisely to the intersection of the vertex set of $\ZZZ(\rho_\delta)$ with the fundamental alcove. This can be seen from the explicit description of $A_\delta$ given in Proposition \ref{A0}.
\end{rem}

\section{Weight diagrams, cap diagrams and decomposition numbers}

\subsection{Weight diagrams and blocks}
Fix $\delta\in \ZZ$. In this section we recall the construction of the weight diagram $x_\lambda$ associated to any partition $\lambda$ given in \cite{CD}.\\
Recall from Section 2.2 that $e_\delta(\lambda)$ is a strictly decreasing sequence in $\ZZ$ for $\delta$ even , and in $\frac{1}{2}+\ZZ$ for $\delta$ odd. 
The weight diagram $x_{\lambda}$ has
vertices indexed by $\NN_0$ if $\delta$ is even or by $\NN-\frac{1}{2}$ if $\delta$ is odd. Each vertex will be
labelled with one of the symbols $\circ$, $\times$, $\vee$, $\wedge$. For $e_\delta(\lambda)=(x_1, x_2, x_3, \ldots)$ define
$$I_{\wedge}(e_\delta(\lambda))=\{x_i: x_i>0\}\quad \text{and}\quad I_{\vee}(e_\delta(\lambda))=\{|x_i|:
x_i<0\}.$$ Now vertex $n$ in
the weight diagram $x_\lambda$ is labelled by
$$
\left\{\begin{array}{cl}
\circ &\text{if}\ n\notin I_{\vee}(e_\delta(\lambda))\cup I_{\wedge}(e_\delta(\lambda))\\
\times&\text{if}\ n\in I_{\vee}(e_\delta(\lambda))\cap I_{\wedge}(e_\delta(\lambda))\\
\vee&\text{if}\ n\in I_{\vee}(e_\delta(\lambda))\setminus I_{\wedge}(e_\delta(\lambda))\\
\wedge&\text{if}\ n\in I_{\wedge}(e_\delta(\lambda))\setminus
I_{\vee}(e_\delta(\lambda))
\end{array}\right.
$$
Moreover, if $x_i=0$ for some $i$ then we label vertex $0$ by either $\wedge$ or $\vee$ (this choice will not affect what follows). Otherwise we label vertex $0$ by  a $\circ$.

\begin{example}\label{empty}
We have
$$e_\delta(\emptyset)=(-\frac{\delta}{2}, -\frac{\delta}{2}-1,-\frac{\delta}{2}-2, -\frac{\delta}{2}-3, -\frac{\delta}{2}-4, \ldots ).$$
Write $\delta = 2m$ or $2m+1$ for some $m\in \mathbb{Z}$. If $m\geq 0$ then $x_\emptyset$ has the first $m$ vertices labelled by $\circ$ and all remaining vertices labelled by $\vee$. If $m<0$ and $\delta$ is odd then $x_\emptyset$ has the first $-m$ vertices labelled by $\times$ and all remaining vertices labelled by $\vee$. Finally if $m<0$ and $\delta$ is even then the first vertex is labelled by $\vee$ (or $\wedge$), the next $-m$ vertices are labelled by $\times$ and all remaining vertices are labelled by $\vee$. The picture is given in Figure \ref{emptypartition}. 
\end{example}

\begin{figure}[ht]
\includegraphics[width=7cm]{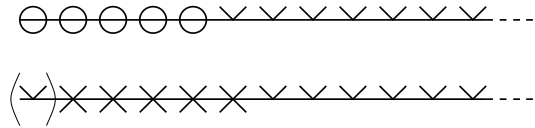}
\caption{The weight diagram $x_\emptyset$ for $\delta=2m$ or $2m+1$ with  $m=5$ or $m=-5$}
\label{emptypartition}
\end{figure}

\noindent It is immediate to see that the $\delta$-degree of singularity of $\lambda$ is equal to the number of $\times$'s in $x_\lambda$.  Note also that  the weight diagram  $x_{\lambda}$ is
labelled by $\vee$ for all  vertices $n>>0$. 

\medskip

\noindent It 
was shown in \cite{CDM2} that two simple 
$B_n(\delta)$-modules 
$L_n(\lambda)$ and $L_n(\mu)$ are in the same block if and only if $e_\delta(\lambda) \in We_\delta(\mu)$. Now it's easy to see that this is equivalent to saying that $x_\lambda$ is obtained from $x_\mu$ by repeatedly applying the following operation:\\
- swapping a $\vee$ and a $\wedge$,\\
- replacing two $\vee$'s (resp. $\wedge$'s) by two $\wedge$'s (resp. $\vee$'s).\\
For $\lambda\in \Lambda$ we define $\mathcal{B}(\lambda)$ to be the set of partitions in the $W$-orbit of $\lambda$ (via the embedding $e_\delta$). Then we have 
$$\mathcal{B}(\lambda)=\cup_m \mathcal{B}_m(\lambda)$$
where $\mathcal{B}_m(\lambda)$ is the block of $B_m(\delta)$ containing $\lambda$ and the union is taken over all $m$ such that $\lambda\in \Lambda_m$.

With $\delta$ still fixed, we now refine the order on $\Lambda$ given by the size of partitions to get the following \textbf{partial order} $\mathbf{\leq}$. We set $\lambda<\mu$ if $x_\mu$ is obtained from $x_\lambda$ by swapping a $\vee$ and a $\wedge$ so that the $\wedge$ moves to the right, or if $x_\mu$ contains a pair of $\wedge$'s instead of a corresponding pair of $\vee$'s in $x_\lambda$, and extending by transitivity. Note that if $\lambda \leq \mu$ then we have $|\lambda | \leq |\mu|$ and $\mu\in \mathcal{B}(\lambda)$.

\subsection{Some properties of the weight diagrams}

Here we give some properties of the weight diagrams which will be used in sections 4 and 5.

\begin{lem}\label{onebox}
Let $\lambda, \mu\in \Lambda$ with $\mu \triangleright \lambda$. Then we have either ${\rm deg}_\delta(\mu)={\rm deg}_\delta(\lambda)$ or ${\rm deg}_\delta (\mu)={\rm deg}_\delta(\lambda)\pm 1$. Moreover, the labels of the weight diagrams $x_\lambda$ and $x_\mu$ are equal everywhere except in at most two adjacent vertices, where we have one of the following configurations.

\smallskip

\noindent \textbf{Case I}: ${\rm deg}_\delta (\mu)={\rm deg}_\delta (\lambda)$ and either\\
(i) $x_\lambda = \times \vee$, $x_\mu = \vee \times$, \, or \qquad
(ii) $x_\lambda = \wedge \times$, $x_\mu = \times \wedge$, \, or \\
(iii) $x_\lambda = \circ \vee$, $x_\mu =  \vee\circ$, \, or \qquad 
(iv) $x_\lambda = \wedge \circ$, $x_\mu = \circ \wedge$,\, or\\
(v) $\delta$ is odd, and vertex $\frac{1}{2}$ is labelled by $\vee$  in $x_\lambda$ and by $\wedge$ in $x_\mu$.

\smallskip

\noindent \textbf{Case II}: ${\rm deg}_\delta(\mu)={\rm deg}_\delta(\lambda)+1$ and either \\
(vi) $x_\lambda = \wedge \vee$, $x_\mu = \circ \times$,\,  or \qquad 
(vii) $x_\lambda = \wedge \vee$, $x_\mu = \times \circ$.

\smallskip

\noindent \textbf{Case III}: ${\rm deg}_\delta(\mu)={\rm deg}_\delta(\lambda)-1$ and either\\
(viii) $x_\lambda = \circ \times$, $x_\mu = \vee \wedge$, \, or \qquad
(ix) $x_\lambda = \times \circ$, $x_\mu = \vee \wedge$.
\end{lem}

\begin{proof}
This follows directly from the definition of the weight diagram.
\end{proof}

\noindent The next lemma will play a key role in what follows.

\begin{lem}\label{key}
Let $\delta =2m$ or $2m+1$ for some $m\in \mathbb{Z}$ and let $\lambda$ be any partition. Then we have
$$\#(\mbox{$\circ$ in $x_\lambda$})-\#(\mbox{$\times$ in $x_\lambda$})=m.$$
\end{lem} 

\begin{proof} We will prove this by induction on the size of $\lambda$.
The result is clearly true for $\lambda = \emptyset$ by Example \ref{empty}.
Now suppose that the result holds for a partition $\lambda$ and let $\mu$ be a partition obtained by adding one box to $\lambda$. Then it is easy to observe from Lemma \ref{onebox} that the result holds for $\mu$. 
\end{proof}

We now explain two ways of recovering the (Young diagram of) the partition $\lambda$ from its weight diagram $x_\lambda$.

First ignore all the $\circ$'s (but not their positions). Now read all the symbols below the line successively from right to left, then all the symbols above the line successively from left to right as illustrated in Figure \ref{weighttranspose}.

\begin{figure}[ht]
\includegraphics[width=10cm]{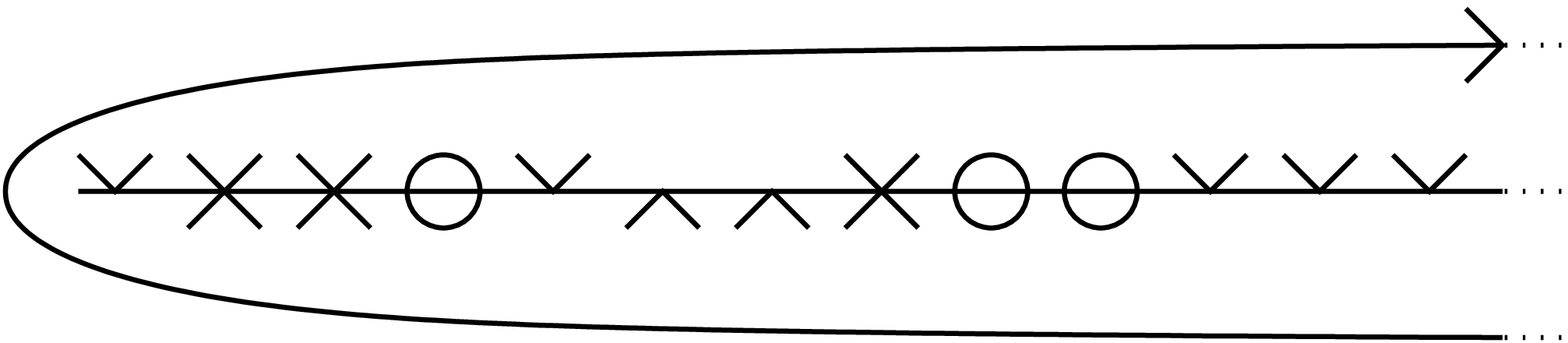}
\caption{}
\label{weighttranspose}
\end{figure}

Now it's easy to see that the Young diagram of the partition $\lambda$ can be drawn as follows. At each entry, if there is no symbol go one step up, and if there is a symbol then go one step to the right. Note that the weight diagram $x_{\lambda}$ ends with infinitely many $\vee$'s. So we always start with infinitely many steps up (the left edge of the quadrant in which $\lambda$ lives) and end with infinitely many steps to the right (the top edge of the quadrant in which $\lambda$ lives), as expected.
Note also that, for $\delta$ even, the entry indexed by $0$ should only be read once as a step up if it is labeled by $\circ$ and a step to the right otherwise. 

\bigskip

Alternatively, one could read the entries from right to left above the line first and then from left to right below the line, as illustrated in Figure \ref{weightpartition}.

\begin{figure}[ht]
\includegraphics[width=10cm]{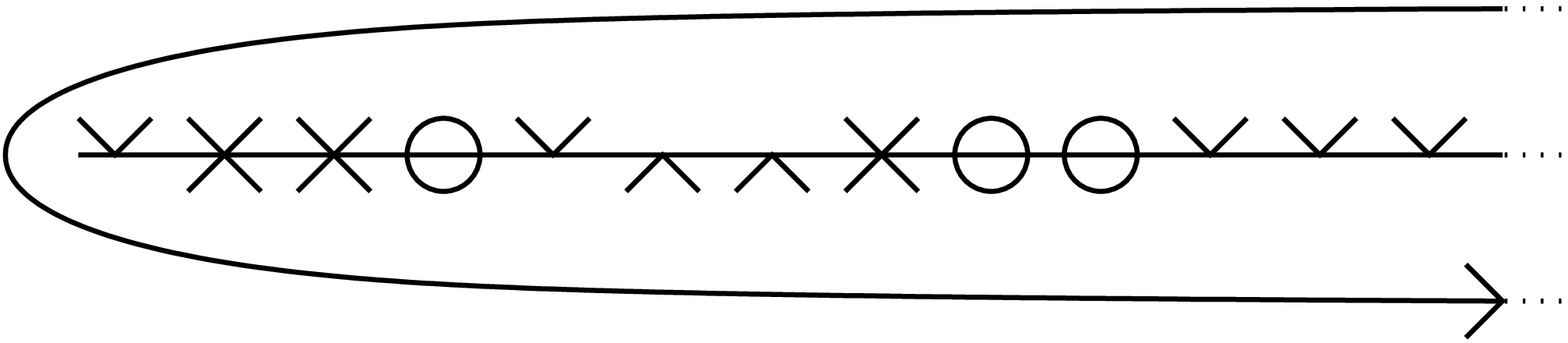}
\caption{}
\label{weightpartition}
\end{figure}

In this case, the Young diagram can be drawn by going one step to the left if there is a symbol and one step down otherwise.

The partition corresponding to the weight diagram in Figure \ref{weighttranspose} and \ref{weightpartition} is given by $\lambda=(10,10,9,9,8,5,3,3)$, see Figure \ref{partition}.

\begin{figure}[ht]
\includegraphics[width=10cm]{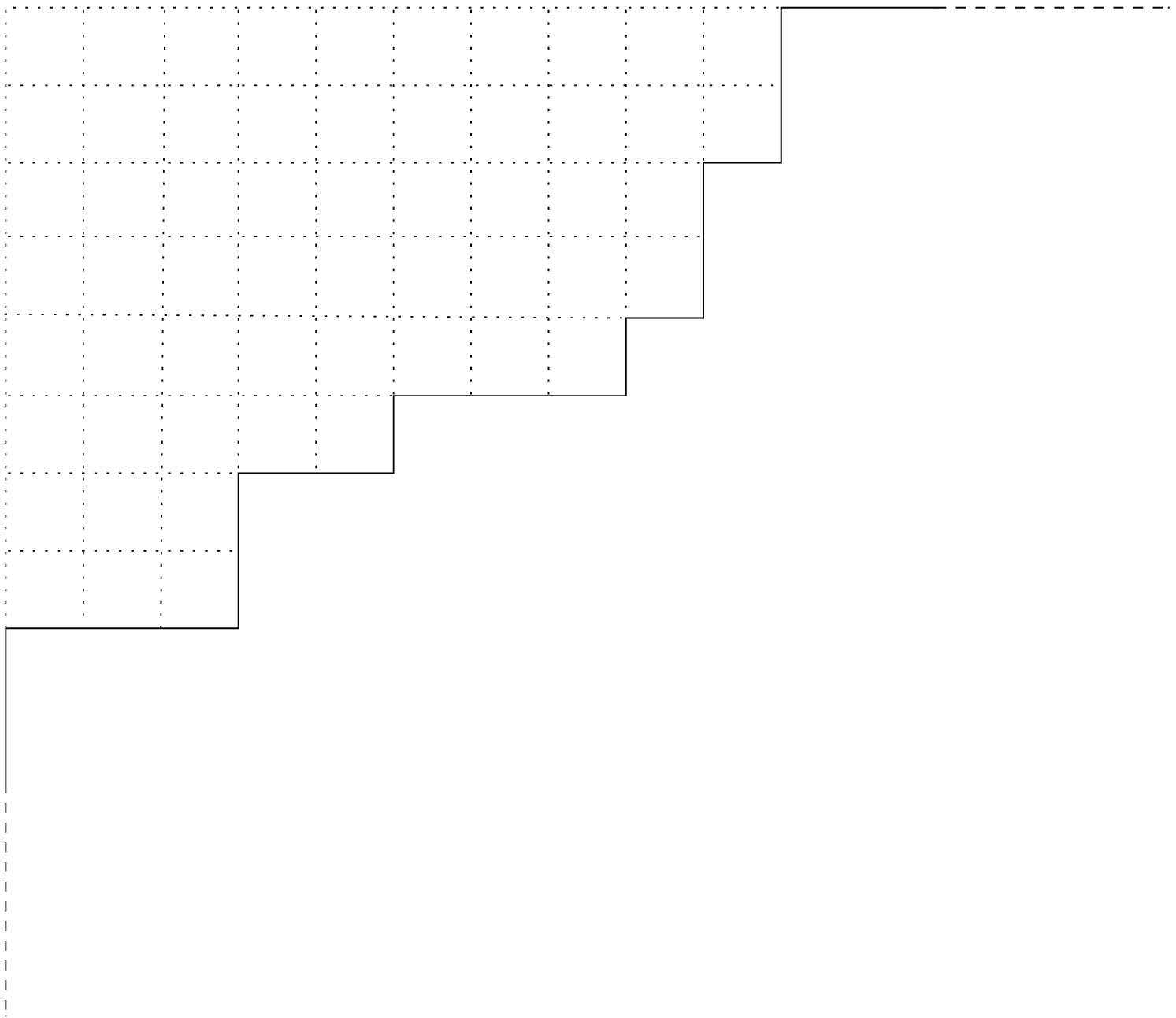}
\caption{}
\label{partition}
\end{figure}

\begin{prop}\label{ylambda} Let $\delta\in \mathbb{Z}$ and $\lambda \in \Lambda$.\\
(i) There is a bijection between the set of $\times$'s in $x_\lambda$ one-to-one and the set of all pairs $(i>j)\in \mathbb{N}^2$ satisfying
$$\lambda^T_i +\lambda^T_j-i-j+2=\delta.$$
Moreover, we have that $(i,j)\in [\lambda]$ if and only if 
\begin{equation}\label{admissibletimes}
\#(\mbox{$\times$'s left of $n$})-\#(\mbox{$\circ$'s left of $n$})<\left\{ \begin{array}{ll} m-1 & \mbox{if $\delta=2m$}\\ m & \mbox{if $\delta=2m+1$} \end{array}\right.
\end{equation}\\
(ii) There is a bijection between the set of $\circ$'s in $x_\lambda$ and the set of all pairs $(i\leq j)\in \mathbb{N}^2$ satisfying
$$-\lambda_i-\lambda_j+i+j=\delta.$$ 
Moreover, we have that $(i,j)\in [\lambda]$ if and only if 
\begin{equation}\label{admissiblecirc}
\#(\mbox{$\times$'s left of $n$})-\#(\mbox{$\circ$'s left of $n$})>\left\{ \begin{array}{ll} m-1 & \mbox{if $\delta=2m$}\\ m & \mbox{if $\delta=2m+1$} \end{array}\right.
\end{equation}
\end{prop}

\begin{proof}
(i) The first part follows from the definition of $x_\lambda$.\\
Now we have $(i,j)\in [\lambda]$ if and only if $i\leq \lambda^T_j$, that is $\lambda^T_j-i\geq 0$. 

\begin{figure}[ht]
\includegraphics[width=10cm]{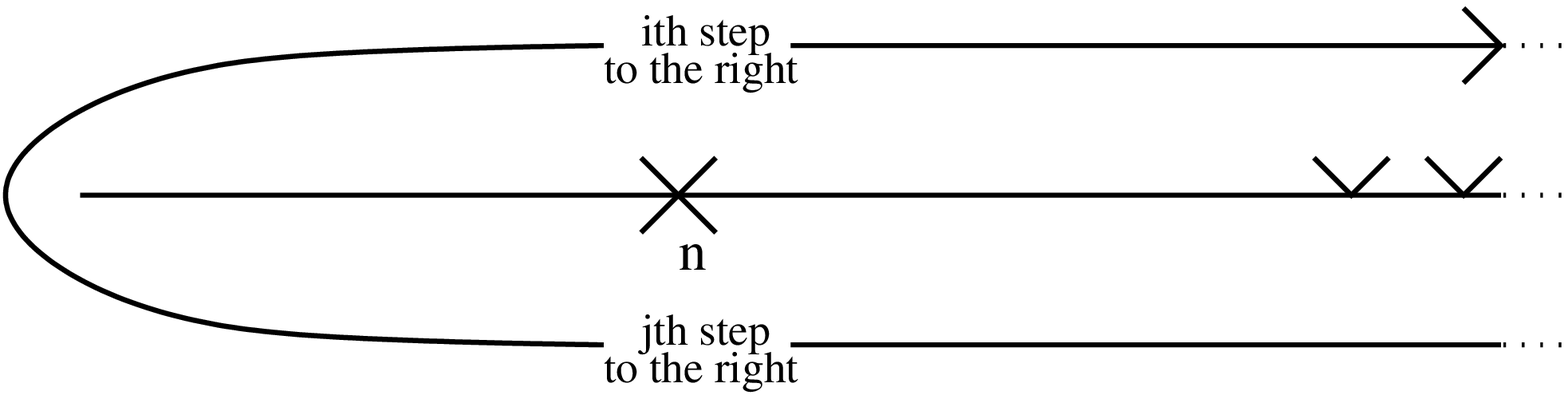}
\caption{}
\label{times}
\end{figure}
Reading the partition from the weight diagram as in the Figure \ref{times} we obtain
\begin{eqnarray*}
\lambda^T_j&=& \# (\mbox{steps up after $j$})\\
&=& \#(\mbox{$\circ$'s left of $n$}) + \#(\mbox{$\vee$'s left of $n$}) + \#(\mbox{$\circ$'s}) + \#(\mbox{$\wedge$'s}) -\delta_{\delta, 2m}, \qquad \mbox{and}\\
i&=& \# (\mbox{steps to the right up to (and including) $i$})\\
&=& \#(\mbox{$\times$'s}) + \#(\mbox{$\wedge$'s}) + \#(\mbox{$\vee$'s left of $n$}) + \#(\mbox{$\times$'s left of $n$}) +1.
\end{eqnarray*}
So we have
\begin{eqnarray*}
\lambda^T_j-i &=& \#(\mbox{$\circ$'s}) - \#(\mbox{$\times$'s}) +\#(\mbox{$\circ$'s left of $n$}) - \#(\mbox{$\times$'s left of $n$})-1-\delta_{\delta,2m}\\
&=& m+\#(\mbox{$\circ$'s left of $n$}) - \#(\mbox{$\times$'s left of $n$})-1-\delta_{\delta,2m} \qquad \mbox{using Lemma \ref{key}}.
\end{eqnarray*}
Hence we have that $\lambda^T_j-i\geq 0$ if and only if
$$\#(\mbox{$\times$'s left of $n$}) - \#(\mbox{$\circ$'s left of $n$}) <m-\delta_{\delta, 2m}$$
as required.

(ii) Suppose that $n$ is a vertex labelled with $\circ$ in $x_\lambda$. Then reading the partition $\lambda$ as in Figure \ref{ijstepdown}, this corresponds to the i-th and j-th steps down. Note that when $n=0$ or $\frac{1}{2}$ we have $i=j$.

\begin{figure}[ht]
\includegraphics[width=10cm]{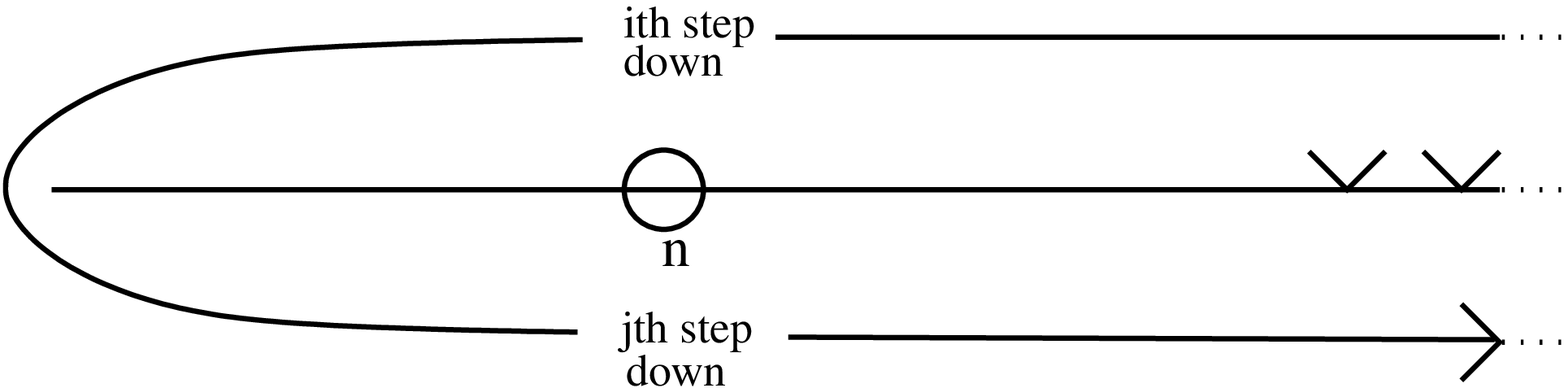}
\caption{}
\label{ijstepdown}
\end{figure}

We want to show first that $n=\lambda_i-i+\frac{\delta}{2}=-(\lambda_j-j+\frac{\delta}{2})$.
Now if we let $k$ be the number of vertices strictly to the left of $n$, then we have that $k=n$ for $\delta$ even and $k=n-\frac{1}{2}$ for $\delta$ odd. Thus if we write $\delta=2m$ or $2m+1$ for some $m\in \ZZ$ then this is equivalent to showing that
\begin{equation}\label{lambdai}
\lambda_i-i=k-m, 
\end{equation}
and
\begin{equation}\label{lambdaj}
\lambda_j-j=\left\{ \begin{array}{ll} -k-m & \mbox{if $\delta=2m$}\\
-k-m-1 & \mbox{if $\delta =2m+1$}\end{array}\right.
\end{equation}
Now we have
\begin{eqnarray*}
\lambda_i &=& \#(\mbox{steps to the left after the ith step down})\\
&=& \# (\mbox{$\times$'s left of $n$}) + \# (\mbox{$\vee$'s left of $n$})
 + \#(\mbox{$\wedge$'s})+\# (\mbox{$\times$'s}), \quad \mbox{and}\\
i &=& \#(\mbox{steps down up to and including $i$})\\
&=& \#(\mbox{$\circ$'s right of $n$}) + \#(\mbox{$\wedge$'s right of $n$}) +1
\end{eqnarray*}
So this gives
\begin{eqnarray*}
\lambda_i-i &=& \# (\mbox{$\times$'s left of $n$}) +\# (\mbox{$\vee$'s left of $n$}) + \# (\mbox{$\wedge$'s left of $n$}) + \# (\mbox{$\circ$'s left of $n$})\\
&& + \# (\mbox{$\times$'s }) - \# (\mbox{$\circ$'s right of $n$}) - \# (\mbox{$\wedge$'s right of $n$}) \\
&& +\# (\mbox{$\wedge$'s right of or at $n$}) -\# (\mbox{$\circ$'s left of $n$}) -1\\
&=& k+\# (\mbox{$\times$'s}) - \#(\mbox{$\circ$'s})\\
&=& k-m \qquad \mbox{using Lemma \ref{key}}.
\end{eqnarray*}
This proves (\ref{lambdai}).\\
Similarly we have
\begin{eqnarray*}
\lambda_j &=& \#(\mbox{steps to the left after the jth step down})\\
&=& \# (\mbox{$\times$'s right of $n$}) + \# (\mbox{$\wedge$'s right of $n$})
  \quad \mbox{and}\\
j &=& \#(\mbox{steps down up to and including $j$})\\
&=& \#(\mbox{$\circ$'s}) + \#(\mbox{$\wedge$'s }) + \#(\mbox{$\vee$'s left of $n$}) +\#(\mbox{$\circ$'s left of $n$}) + 1 - \delta_{\delta,2m}
\end{eqnarray*}
where the term -$\delta_{\delta,2m}$ cancels the double counting of the vertex indexed by $0$ in the even $\delta$ case.\\
So this gives
\begin{eqnarray*}
\lambda_j-j &=& -\# (\mbox{$\vee$'s  left of $n$}) -\# (\mbox{$\circ$'s left of $n$}) - \# (\mbox{$\wedge$'s left of $n$}) - \# (\mbox{$\times$'s left of $n$})\\
&& + \# (\mbox{$\times$'s }) - \# (\mbox{$\circ$'s }) -1 +\delta_{\delta,2m}\\
&=& \left\{ \begin{array}{ll} -k -m & \mbox{if $\delta=2m$}\\ -k-m-1 & \mbox{if $\delta=2m+1$} \end{array}\right. \qquad \mbox{using Lemma \ref{key}}
\end{eqnarray*} 
proving (\ref{lambdaj}).\\
Moreover, we have $(i,j)\in [\lambda]$ if and only if $j\leq \lambda_i$, that is $\lambda_i-j\geq 0$. Now we get
\begin{eqnarray*}
\lambda_i-j &=& \#(\mbox{$\times$'s}) - \#(\mbox{$\circ$'s}) +\#(\mbox{$\times$'s left of $n$}) - \#(\mbox{$\circ$'s left of $n$})-\delta_{\delta,2m+1}\\
&=& -m+\#(\mbox{$\times$'s left of $n$}) - \#(\mbox{$\circ$'s left of $n$})-\delta_{\delta,2m+1} \qquad \mbox{using Lemma \ref{key}}.
\end{eqnarray*}
Hence we have that $\lambda_i-j\geq 0$ if and only if
$$\#(\mbox{$\times$'s left of $n$}) - \#(\mbox{$\circ$'s left of $n$}) >m-\delta_{\delta, 2m}$$
as required.
\end{proof}

\subsection{Cap diagrams and decomposition numbers}

In this section we associate an oriented cap diagrams $c_\lambda$ to any partition $\lambda$ (and $\delta\in \mathbb{Z}$) as in \cite{CD}. Note that this is slight reformulation of the Temperley-Lieb half diagrams associated to $\lambda$ given in \cite{M}, but we keep the information about the positions of the $\times$'s and $\circ$'s. A similar construction was also given in \cite{Le}.

First, draw the vertices of $x_\lambda$ on the horizontal edge of the NE quadrant of the plane.
Now, in $x_{\lambda}$ find a pair of vertices labelled $\vee$ and $\wedge$
in order from left to right that are neighbours in the sense that
there are only separated by $\circ$s, $\times$s or vertices already joined by a cap. Join this pair of vertices together
with a cap. Repeat this process until there are no more such $\vee$
$\wedge$ pairs. (This will occur after a finite number of steps.)

Ignoring all $\circ$s, $\times$s and vertices on a cap, we are left
with a sequence of a finite number of $\wedge$s followed by an
infinite number of $\vee$s. Starting from the leftmost $\wedge$, join
each $\wedge$ to the next from the left which has not yet been used,
by a cap touching the vertical boudary of the NE quadrant, without crossing any other caps. If there is a free
$\wedge$ remaining at the end of this procedure, draw an infinite ray
up from this vertex, and draw infinite rays from each of the remaining
$\vee$s.

\noindent Examples of this construction are given in Figure \ref{caps}.

Here we have drawn the \lq curls' from \cite{CD} as caps touching the edge of the NE quadrant, as this is better suited for the combinatorics introduced in Section 3.

\medskip

By \cite{M}, we have that the decomposition numbers
$D_{\lambda\mu}=[\Delta_n(\mu):L_n(\lambda)]=(P_{n}(\lambda):\Delta_{n}(\mu))$ can be described using these cap diagrams as follows. 
Define the polynomial $d_{\lambda\mu}(q)$ by setting $d_{\lambda\mu}(q)\neq 0$ if and only if $x_{\mu}$ is obtained from
$x_{\lambda}$ by changing the labellings of the elements in some of the
pairs of vertices joined by a cap in $c_\lambda$ from $\vee$ $\wedge$ to $\wedge$ $\vee$ or from $\wedge$ $\wedge$ to $\vee$ $\vee$. In that
case define $d_{\lambda\mu}(q)=q^k$ where $k$ is the number of pairs whose
labellings have been changed. Then we have
\begin{equation}\label{decnumbers}
D_{\lambda\mu}=d_{\lambda\mu}(1).
\end{equation}

\section{Restriction of simple modules}
Here we describe completely the module structure of the restriction  of $L_n(\lambda)$ to $B_{n-1}(\delta)$ for any $\lambda\in \Lambda_n$. 
For $\lambda \in \Lambda$ we define $\supp(\lambda)=\{\mu \in \Lambda \, | \, \mu\triangleright \lambda \,\, \mbox{or} \,\, \mu \triangleleft \lambda\}$. Recall $\mathcal{B}(\mu)$ is the block containing $\mu\in \Lambda$ (as defined in Section 3.1). We define ${\rm pr}^\mu$ to be the functor projecting onto $\mathcal{B}(\mu)$ and write $\res_n^\mu$ for ${\rm pr}^\mu \circ \res_n$. 
We have that $\res_nL_n(\lambda)$ decomposes as
$${\rm res}_nL_n(\lambda)=\bigoplus_{\mathcal{B}(\mu)} {\rm res}_n^\mu L_n(\lambda),$$
and using (\ref{resDelta}) the direct sum can be taken over all blocks $\mathcal{B}(\mu)$ with $\mathcal{B}(\mu)\cap {\rm supp}(\lambda)\neq \emptyset$.
Thus it is enough to describe ${\rm res}_n^\mu L_n(\lambda)$ for each $\mu\in {\rm supp}(\lambda)\cap \Lambda_{n-1}$.
We have three cases to consider depending on the relative degree of singularity of $\lambda$ and $\mu$ as in Lemma \ref{onebox}.\\
Case I: ${\rm deg}_\delta (\mu)={\rm deg}_\delta(\lambda)$.\\
Case II: ${\rm deg}_\delta (\mu)={\rm deg}_\delta(\lambda)+1$.\\
Case III: ${\rm deg}_\delta (\mu)={\rm deg}_\delta(\lambda)-1$.

\smallskip

Case I has been dealt with in \cite{CDM3}. We state the result here for completeness.

\begin{prop}\label{caseI}\cite{CDM3}(Proposition 4.1)
If $\lambda\in \Lambda_n$ and $\mu \in {\rm supp}(\lambda)\cap \Lambda_{n-1}$ with ${\rm deg}_\delta(\mu)={\rm deg}_\delta(\lambda)$ then we have
$${\rm res}_n^{\mu}L_n(\lambda)=L_{n-1}(\mu).$$
\end{prop}

It will be convenient to define a new notation here for cases II and III. Suppose that $\lambda'\in {\rm supp}(\lambda)$ with ${\rm deg}_\delta (\lambda')={\rm deg}_\delta(\lambda)+1$. Then it's easy to see from Lemma \ref{onebox} and the description of blocks given in Section 3.1 that 
$${\rm supp}(\lambda')\cap \mathcal{B}(\lambda)=\{\lambda^+, \lambda^-\}$$
with one of $\lambda^+$ or $\lambda^-$ being equal to $\lambda$. We can also assume that $\lambda^+>\lambda^-$. Moreover we have that the weight diagrams of 
 $\lambda', \lambda^+$ and $\lambda^-$ differ in precisely two adjacent vertices, say $i-1$ and $i$ as depicted in Figure \ref{basic}. 

\begin{figure}[ht]
\includegraphics[width=12cm]{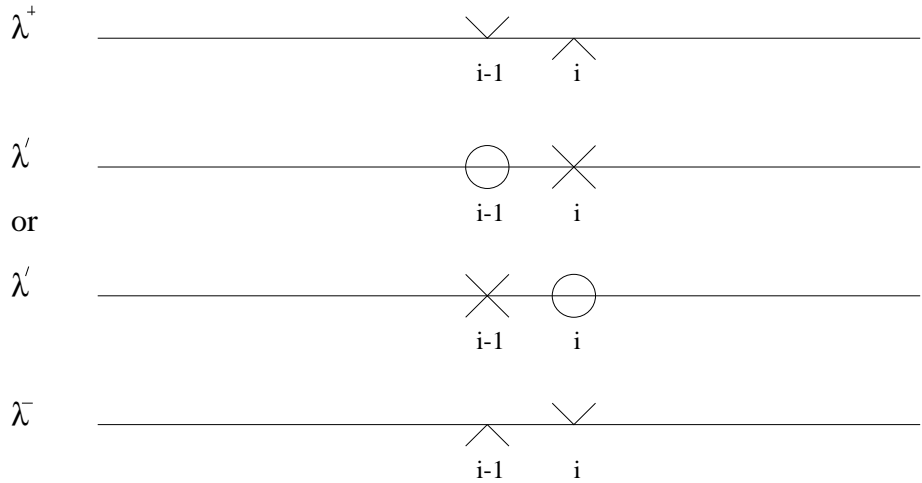}
\caption{Weight diagrams corresponding to $\lambda'$, $\lambda^-$ and $\lambda^+$.}
\label{basic}
\end{figure}

In all the following figures we will always assume that the weight diagram of $\lambda'$ has labels $\circ$ $\times$ in positions $i-1$,$i$. The other case is exactly the same. 

Using the above notation we can now state the result for Case II as follows.

\begin{prop}\label{caseII} \cite{CDM3}(Theorem 4.8)
Let $\lambda', \lambda^+, \lambda^-$ be as above then we have
\begin{eqnarray*}
&(i)& {\rm res}_n^{\lambda'}L_n(\lambda^+)=L_{n-1}(\lambda'), \quad \mbox{and}\\
&(ii)& {\rm res}_n^{\lambda'}L_n(\lambda^-)=0.
\end{eqnarray*}
\end{prop}

Keeping the same notation, for Case III we need to describe ${\rm res}_n^{\lambda^+}L_n(\lambda')$.
This is more complicated than the previous two cases. Indeed we will see that the number of composition factors can get arbitrarily large (as $n$ varies).

Note that for any $\mu'\in \mathcal{B}(\lambda')$ we have ${\rm supp}(\mu')\cap \mathcal{B}(\lambda^+)=\{\mu^+, \mu^-\}$ and ${\rm supp}(\mu^{\pm})\cap \mathcal{B}(\lambda')=\{\mu' \}$. The weight diagrams of $\mu', \mu^+$ and $\mu^-$ differ in precisely vertices $i-1$ and $i$ and these two vertices are labelled as in $\lambda', \lambda^+$ and $\lambda^-$ respectively.

\medskip

We now recall a result from \cite{M}(proof of (7.7)) which will be needed in the proof of the next theorem.

\begin{prop}\label{indP} Let $\lambda', \lambda^+, \lambda^-$ and $\mu', \mu^+, \mu^-$ be as above. Suppose $\mu'\in \Lambda_{n}$. Then we have
\begin{eqnarray*}
&&{\rm ind}_{n-1}^{\lambda'}P_{n-1}(\mu^-)\cong P_n(\mu'), \qquad \mbox{and}\\
&&{\rm ind}_{n-1}^{\lambda'}P_{n-1}(\mu^+)\cong 2 P_n(\mu').
\end{eqnarray*}
\end{prop}

Note that the cap diagram associated to a partition splits the NE quadrant of the plane into open connected components, called chambers. We say that a vertex a cap, or a ray belongs to a chamber $C$ if it is in the closure of $C$. Note that each vertex labelled with $\times$ or $\circ$ belongs to precisely one chamber and each vertex labelled $\vee$ or $\wedge$ and each cap or ray belongs to precisely two chambers. In  $x_{\lambda'}$, the vertex $i$ is labelled with $\times$ or $\circ$, so it belongs to a unique chamber $C_i$ in the cap diagram of $c_{\lambda'}$. 

Now we define the subset $I(\lambda', \lambda^+)$ of the set of vertices of $c_{\lambda'}$ by setting $j\in I(\lambda',\lambda^+)$ if and only if $j$ belongs to $C_i$ and one of the following three possibilities holds
\begin{equation}\label{Ii}
\mbox{$j>i$ and  $j$ is labelled with $\vee$,}
\end{equation}
\begin{equation}\label{Iii}
\mbox{$j<i$ and $j$ is labelled with $\wedge$, or}
\end{equation}
\begin{equation}\label{Iiii}
\mbox{$j<i$ and $j$ is labelled with $\vee$ and it is either on a ray or connected to some $k>i$.}
\end{equation}

\noindent Examples of all $j\in I(\lambda', \lambda^+)$ for various $\lambda'$ are given in Figure \ref{caps}.

For each $j\in I(\lambda', \lambda^+)$, define $\lambda'_{(j)}$ as follows. If $j$ satisfies (\ref{Ii}) above then $\lambda'_{(j)}$ is the partition whose weight diagram is obtained from $x_{\lambda'}$ by labelling vertex $j$ with $\wedge$, and vertices $i-1$ and $i$ with $\vee$, leaving everything else unchanged. If $j$ satisfies case (\ref{Iii}) above then $\lambda'_{(j)}$ is the partition whose weight diagram is obtained from $x_{\lambda'}$ by labelling vertex $j$ with $\vee$, and vertices $i-1$ and $i$ with $\wedge$, leaving everything else unchanged. Finally, if $j$ satisfies case (\ref{Iiii}) then $\lambda'_{(j)}$ is the partition whose weight diagram is obtained from $x_{\lambda'}$ by labelling vertex $j, i$ and $i-1$ with $\wedge$,  leaving everything else unchanged.

\begin{figure}[ht]
\includegraphics[width=12cm]{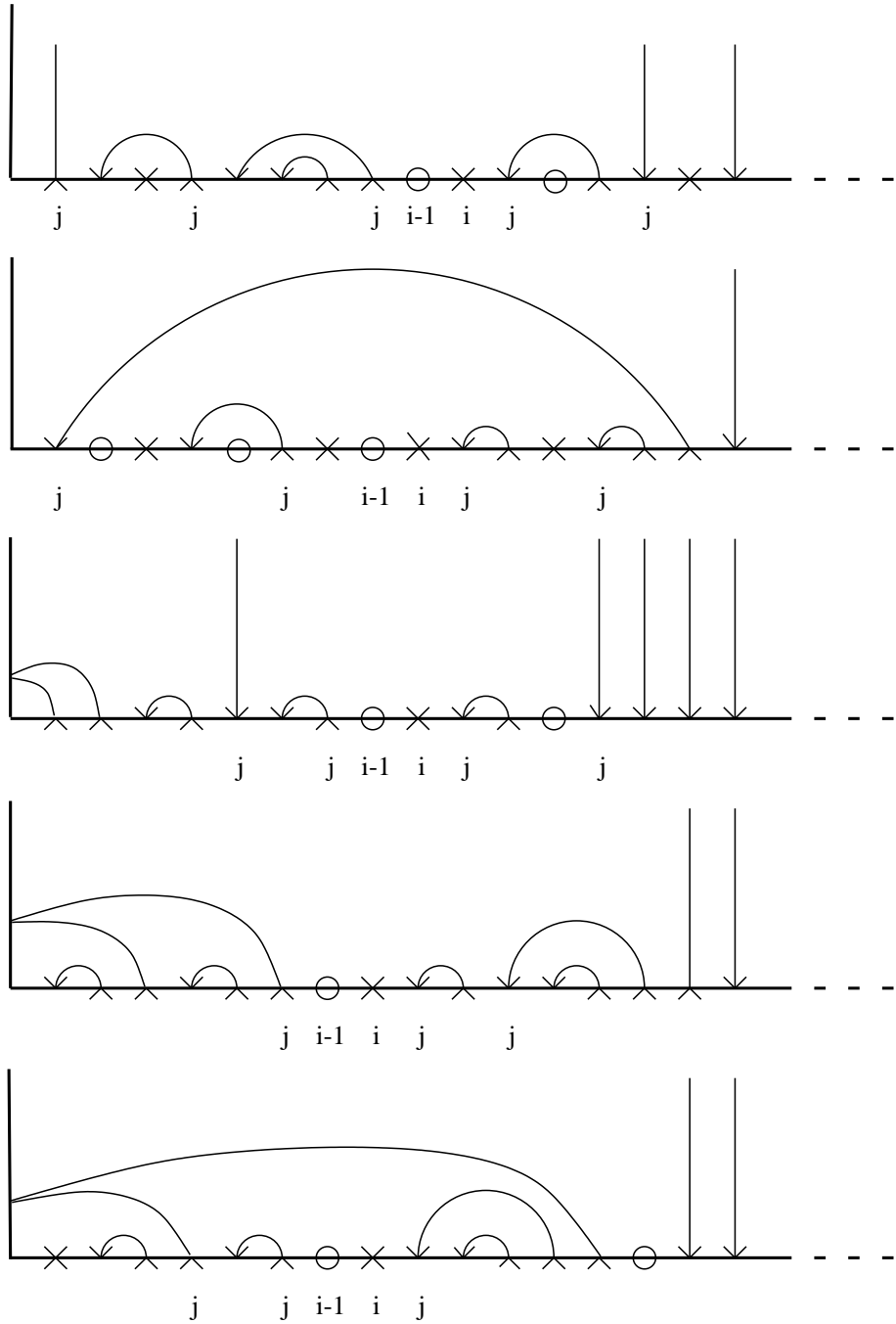}
\caption{Examples of cap diagrams associated to $\lambda'$ with the set of vertices $j\in I(\lambda', \lambda^+)$}
\label{caps}
\end{figure}

\noindent Now define 
$$\Lambda_{n-1}(\lambda', \lambda^+)=\{\lambda'_{(j)}\,\, :\,\, j\in I(\lambda', \lambda^+)\}\cap \Lambda_{n-1}.$$ 

\begin{thm} 
Let $\lambda', \lambda^+, \lambda^-$ be as above.\\
 If $|\lambda'|=n$ then we have $\res_n^{\lambda^+}L_n(\lambda')=L_{n-1}(\lambda^-)$.\\
If $|\lambda'|<n$ then we have that $\res_n^{\lambda^+}L_n(\lambda')$ has simple head and simple socle isomorphic to $L_{n-1}(\lambda^+)$ and we have
$$\rad(\res_n^{\lambda^+}L_n(\lambda'))/\soc(\res_n^{\lambda^+}L_n(\lambda'))
\, \cong \, L_{n-1}(\lambda^-) \oplus \bigoplus_{\mu\in \Lambda_{n-1}(\lambda',\lambda^+)}L_{n-1}(\mu) $$
\end{thm}

\begin{proof}
If $|\lambda'|=n$ we have $L_n(\lambda')=\Delta_n(\lambda')$ and $L_{n-1}(\lambda^-)=\Delta_{n-1}(\lambda^-)$ so the result follows immediately from the exact sequence given in (\ref{resDelta}).\\
Now suppose that $|\lambda'|<n$. Let us start by finding the composition factors of this module.
Note that for any $\mu\in \calB(\lambda^+)$ we have
\begin{eqnarray*}
[\res_n^{\lambda^+}L_n(\lambda'):L_{n-1}(\mu)]&=& \dim \Hom_{n-1} (P_{n-1}(\mu),\res_n^{\lambda^+}L_n(\lambda'))\\
&=& \dim \Hom_n (\ind_{n-1}^{\lambda'}P_{n-1}(\mu), L_n(\lambda')).
\end{eqnarray*} 
If $\Hom_n (\ind_{n-1}^{\lambda'}P_{n-1}(\mu), L_n(\lambda'))\neq 0$, then $P_n(\lambda')$ must be a summand of $\ind_{n-1}^{\lambda'}P_{n-1}(\mu)$. 
So we must have $(P_{n-1}(\mu):\Delta_{n-1}(\lambda^+))\neq 0$ or $(P_{n-1}(\mu):\Delta_{n-1}(\lambda^-))\neq 0$ using (\ref{indDelta}). The $\Delta$-factors of the projective module $P_{n-1}(\mu)$ are given by (\ref{decnumbers}). 
As $\mu\in \mathcal{B}(\lambda^+)$, using the block description on weight diagrams given in Section 2.1, we have that the vertices $i-1$ and $i$ in $x_\mu$ are labelled by $\vee$ or $\wedge$. Note that we must also have either $\mu\geq \lambda^+$ or $\mu\geq \lambda^-$. We now have four cases to consider, depending on the labellings of $i-1$ and $i$.

\smallskip

\noindent \textbf{Case A}. In $x_\mu$, vertices $i-1$ and $i$ are labelled by $\vee$ and $\wedge$ respectively. Here $\mu$ is of the from $\mu^+$ and so from Proposition \ref{indP} we have $\ind_{n-1}^{\lambda'}P_{n-1}(\mu)=2P_n(\mu')$. Thus we must have $\mu=\lambda^+$ and we have
$$\dim \Hom_n(\ind_{n-1}^{\lambda'}P_{n-1}(\lambda^+), L_n(\lambda'))= 2.$$

\medskip

\noindent \textbf{Case B}. In $x_\mu$, vertices $i-1$ and $i$ are labelled by $\wedge$ and $\vee$ respectively. Here $\mu$ is of the form $\mu^-$ and so, by Proposition \ref{indP} we have $\ind_{n-1}^{\lambda'}P_{n-1}(\mu)=P_n(\mu')$.Thus we must have $\mu=\lambda^-$ and we have
$$\dim \Hom_n(\ind_{n-1}^{\lambda'}P_{n-1}(\lambda^-), L_n(\lambda'))= 1.$$

\medskip

\noindent \textbf{Case C}. In $x_\mu$, both vertices $i-1$ and $i$ are labelled by $\vee$. There are two subcases C(i) and C(ii) to consider here depending on the cap diagram $c_\mu$. First consider the case C(i) as depicted in Figure \ref{caseC(i)}.

\begin{figure}[ht]
\includegraphics[width=8cm]{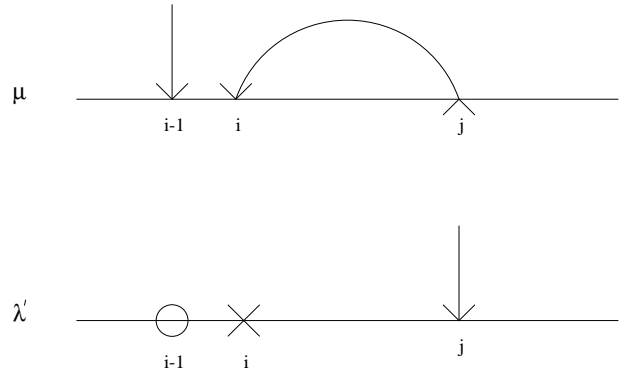}
\caption{Case C(i)}
\label{caseC(i)}
\end{figure}

\noindent Using (\ref{decnumbers}) we have that $(P_{n-1}(\mu):\Delta_{n-1}(\lambda^-))=0$ and if $(P_{n-1}(\mu):\Delta_{n-1}(\lambda^+))\neq 0$ then $x_{\lambda^+}$ is obtained from $x_{\mu}$ by swapping the labelling of vertices $i$ and $j$ and possibly other pairs connected by a cap  in $c_\mu$. Now when we apply the functor $\ind_{n-1}^{\lambda'}$ to $P_{n-1}(\mu)$, it follows from (\ref{indDelta}) that all its $\Delta$-factors corresponding to weight diagrams with the labels of $i$ and $j$ as in $x_\mu$ (namely $\vee$ and $\wedge$ resp.) will go to zero.
Let $\eta$ be the partition obtained from $\mu$ by swapping the labels on vertices $i$ and $j$ and leaving everything else unchanged. Note that $\eta$ is of the form $\eta^+$ and it is the largest $\Delta$-factor not annihilated by $\ind_{n-1}^{\lambda'}$ . Now it's easy to see from (\ref{decnumbers}) that there is a one-to-one correspondence between the $\Delta$-factors of $P_{n-1}(\mu)$ with the labels of $i$ and $j$ swapped (that is, those not annihilated by the induction functor) and the $\Delta$-factors of $P_n(\eta')$. This implies that 
$$\ind_{n-1}^{\lambda'}P_{n-1}(\mu)=P_n(\eta').$$
Thus we must have $\eta'=\lambda'$ with $\lambda'$ as in Figure \ref{caseC(i)}, where vertices $i$ and $j$ are in the closure of the same chamber.
Thus for each such $\mu$ we have $\dim \Hom_n(\ind_{n-1}^{\lambda'}P_{n-1}(\mu), L_n(\lambda'))=1$.

\medskip

Now consider the case C(ii) with $\mu$ as depicted in Figure \ref{caseC(ii)}.

\begin{figure}[ht]
\includegraphics[width=8cm]{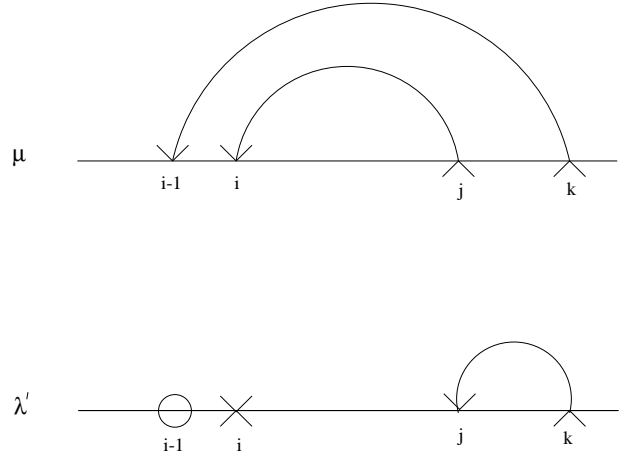}
\caption{Case C(ii)}
\label{caseC(ii)}
\end{figure}

\noindent If $(P_{n-1}(\mu):\Delta_{n-1}(\lambda^\pm))\neq 0$ then $\lambda^\pm$ is obtained from $\mu$ by swapping the labelling of vertices $i$ and $j$ for $\lambda^+$ and of vertices $i-1$ and $k$ for $\lambda^-$ and possibly other pairs connected by a cap  as well. Now when we apply the functor $\ind_{n-1}^{\lambda'}$ to $P_{n-1}(\mu)$, it follows from (\ref{indDelta}) that all its $\Delta$-factors corresponding to diagrams with the labels of $i$ and $j$ and the labels of $i-1$ and $k$ are either both swapped or both unchanged will go to zero.
Let $\eta$ be the partition whose weight diagram is obtained from $x_\mu$ by swapping the labels on vertices $i$ and $j$ and leaving everything else unchanged. Note that $\eta$ is of the form $\eta^+$ and it is the largest $\Delta$-factor not annihilated by $\ind_{n-1}^{\lambda'}$ . Now it's easy to see that there is a one-to-one correspondence between the $\Delta$-factors of $P_{n-1}(\mu)$ not annihilated by the induction functor and the $\Delta$-factors of $P_n(\eta')$. This implies that 
$$\ind_{n-1}^{\lambda'}P_{n-1}(\mu)=P_n(\eta').$$
Thus we must have $\eta'=\lambda'$ with $\lambda'$ as depicted in Figure \ref{caseC(ii)}, where vertices $i$ and $j$ are in the closure of the same chamber.
For each such $\mu$ we have $\dim \Hom_n(\ind_{n-1}^{\lambda'}P_{n-1}(\mu), L_n(\lambda'))=1$.

We have seen that Case C covers all $j\in I(\lambda', \lambda^+)$ satisfying (\ref{Ii}).

\medskip

\noindent \textbf{Case D}. In $x_\mu$ both vertices $i-1$ and $i$ are labelled by $\wedge$. This case splits into six subcases D(i)-(vi) as depicted in Figures \ref{caseD(i)}-\ref{caseD(vi)}. 
Using the same argument as in Case C, it is easy to show that in each case we have $\dim \Hom_n(\ind_{n-1}^{\lambda'}P_{n-1}(\mu), L_n(\lambda'))=1$.
Note that in all cases the vertices $i$ and $j$ in $\lambda'$ must be in the same chamber otherwise $x_\mu$ wouldn't have the required cap diagram. 
We have that cases D(i)(iii)-(v) correspond to all vertices $j$ satisfying (\ref{Iii}), and Cases D(ii) and (vi) correspond to all vertices $j$ satisfying (\ref{Iiii}).

\begin{figure}[ht]
\includegraphics[width=8cm]{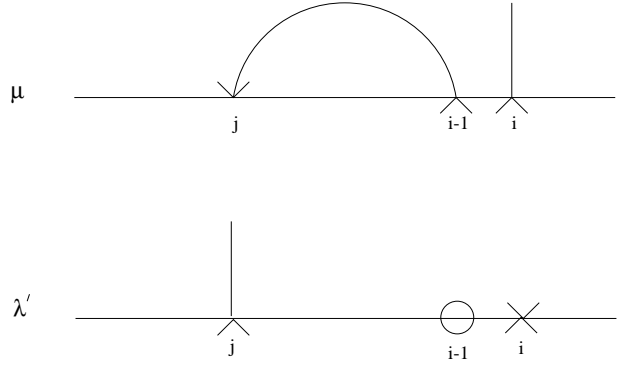}
\caption{Case D(i)}
\label{caseD(i)}
\end{figure}

\begin{figure}[ht]
\includegraphics[width=8cm]{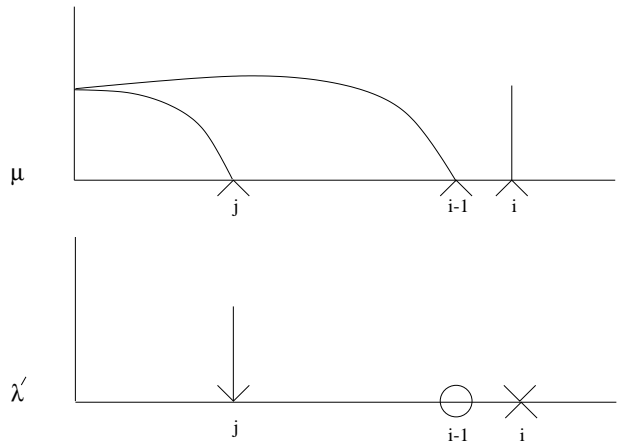}
\caption{Case D(ii)}
\label{caseD(ii)}
\end{figure}

\begin{figure}[ht]
\includegraphics[width=8cm]{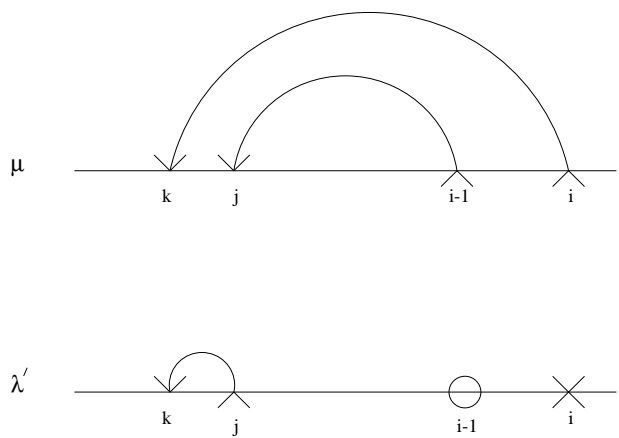}
\caption{Case D(iii)}
\label{caseD(iii)}
\end{figure}

\begin{figure}[ht]
\includegraphics[width=8cm]{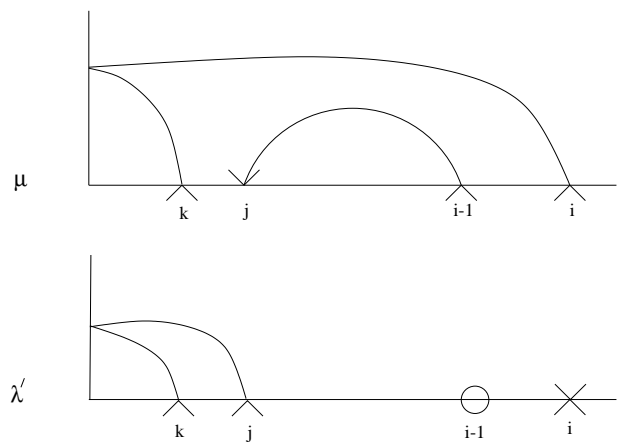}
\caption{Case D(iv)}
\label{caseD(iv)}
\end{figure}

\begin{figure}[ht]
\includegraphics[width=8cm]{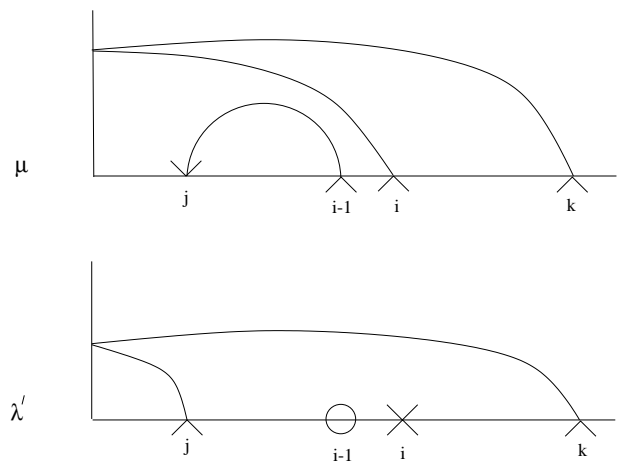}
\caption{Case D(v)}
\label{caseD(v)}
\end{figure}

\begin{figure}[ht]
\includegraphics[width=8cm]{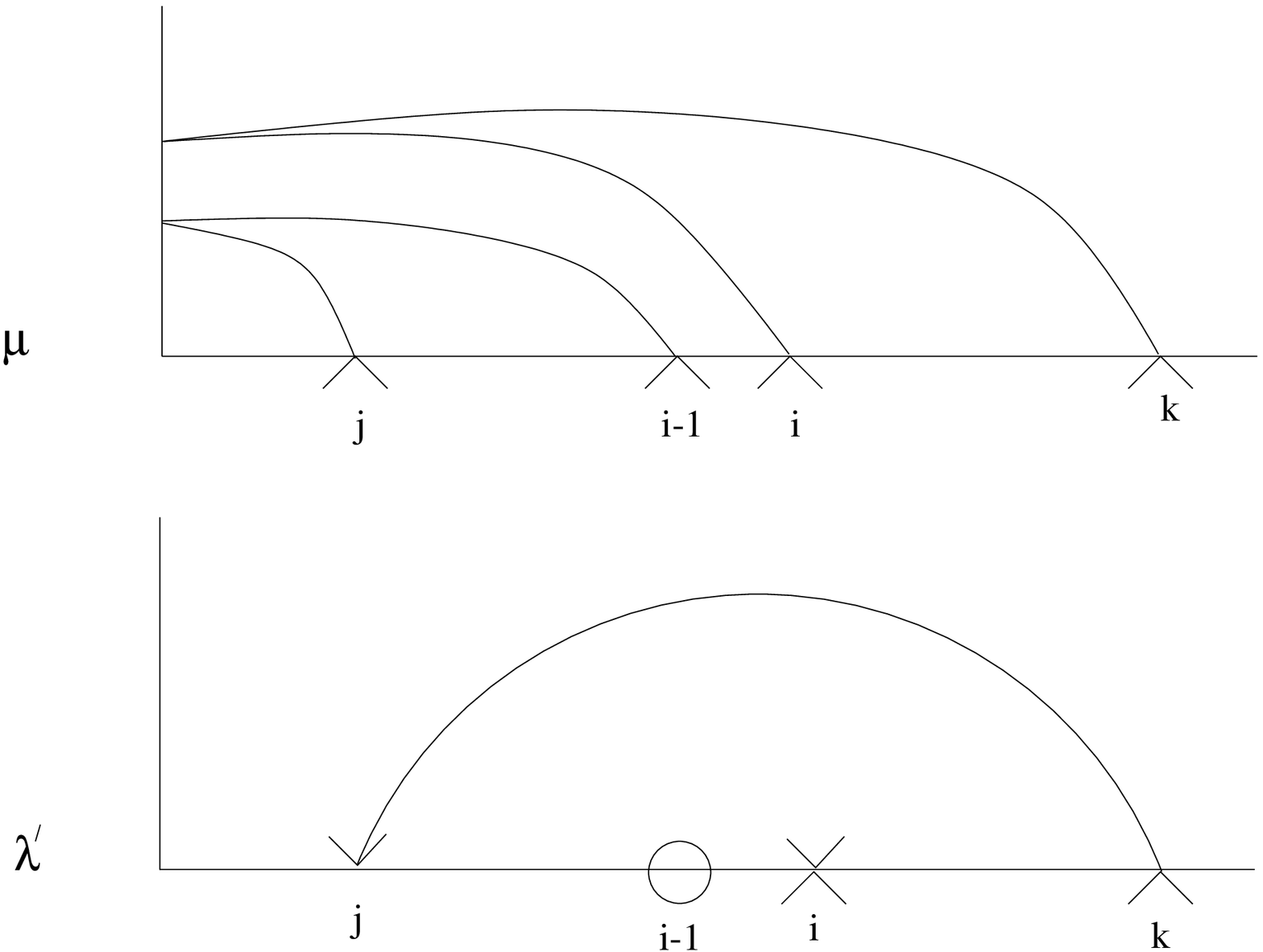}
\caption{Case D(vi)}
\label{caseD(vi)}
\end{figure}

\smallskip

Finally, using the fact that all simple modules are self-dual, the restriction must have the required module structure.
\end{proof} 

\begin{cor}\label{dimsimple} Let $\lambda\in A_\delta$. Then the dimension of $L_n(\lambda)$ is given by the number of walks on $\mathcal{Y}_\delta$ from $\emptyset$ to $\lambda$.

\begin{proof}
This follows immediately by induction on $n$ using Propositions \ref{caseI} and \ref{caseII}.
\end{proof}

\end{cor}

\section{Walk bases for simple modules}

\subsection{Leduc-Ram walk bases for generic simple modules}

In this section we recall the construction given in \cite{LR} of walk bases for simple modules for the generic Brauer algebra $B_n(u)$. Their construction uses two combinatorial objects associated with partitions which we now recall. We start with the King polynomials, which were originally derived from Weyl's character formula in \cite{EK}.

Let $\lambda$ be a partition, and denote by $[\lambda]$ its Young diagram. For each box $(i,j)\in [\lambda]$ we define
$$d(i,j)=\left\{ \begin{array}{ll} \lambda_i + \lambda_j -i-j+1 & \mbox{if $i\leq j$}\\ -\lambda^T_i -\lambda^T_j+i+j-1 & \mbox{if $i>j$}\end{array}\right.$$
We also write $h(i,j)$ for the usual hook length. We then define the King polynomial
$$P_\lambda (u)=\prod_{(i,j)\in [\lambda]}\frac{u-1+d(i,j)}{h(i,j)}.$$
For example, $P_\emptyset(u)=1$, $P_{(1)}(u)=u$, 
$$P_{(1^2)}(u)=\frac{u(u-1)}{2}, P_{(2)}(u)=\frac{(u+2)(u-1)}{2},$$
$$P_{(1^3)}(u)=\frac{u(u-2)(u-1)}{3!}, P_{(2,1)}(u)=\frac{(u+2)u(u-2)}{3}, P_{(3)}(u)=\frac{(u+4)u(u-1)}{3!},$$
$$P_{(1^4)}(u)=\frac{u(u-3)(u-2)(u-1)}{4!}, P_{(2,1^2)}(u)=\frac{(u+2)u(u-3)(u-1)}{4.2},\ldots$$

We denote the set of all walks on the Young graph $\mathcal{Y}$ by $\Omega$ and the subset of all walks of length n starting at $\emptyset$ and ending at $\lambda$ by $\Omega^n(\lambda)$.
For a walk $S\in \Omega$, we write $S=(s(0), s(1), s(2),...)$ where $s(m)$ is the m-th partition in the walk $S$. We then define $\Omega_m(S)$ to be the set of all walks $T$ that differs from $S$ in at most position $m$, that is $t(j)=s(j)$ for all $j\neq m$. If $T\in \Omega_m(S)$ we say that $(S,T)$ form an $m$-diamond pair, and in this case we define
$$\Diamond_m(S,T)= \left\{ \begin{array}{ll} \pm (s(m+1)_k - k - t(m)_l +l) & \begin{array}{l} \mbox{if $t(m)=s(m-1)\pm \epsilon_l$}\\ \mbox{and $s(m+1)=s(m)\pm \epsilon_k$}\end{array}\\
 &  \\
\pm (u + t(m)_l - l +s(m+1)_k -k) & \begin{array}{l}\mbox{if $t(m)=s(m-1)\mp \epsilon_l$} \\ \mbox{ and $s(m+1)=s(m)\pm \epsilon_k$}\end{array}\end{array}\right.$$

\begin{thm} \cite{LR}(6.22) There is an action of the generic Brauer algebra $B_n(u)$ on the complex vector space $\Pi^\lambda$ with basis $\Omega^n(\lambda)$ given by
$$\sigma_m T = \sum_{S\in \Omega_m(T)}(\sigma_m)_{ST}S$$
$$e_m T= \sum_{S\in \Omega_m(T)}(e_m)_{ST}S$$
where
$$(\sigma_m)_{SS}=\left\{ \begin{array}{ll} \frac{1}{\Diamond_m(S,S)} & \mbox{if $s(m-1)\neq s(m+1)$} \\ \frac{1}{\Diamond_m(S,S)} \left( 1 - \frac{P_{s(m)}(u)}{P_{s(m-1)}(u)}\right) & \mbox{otherwise} \end{array} \right.$$
and for $S\neq T$ 
$$(\sigma_m)_{ST}=\left\{ \begin{array}{ll} \sqrt{\frac{(\Diamond_m(S,S)-1)(\Diamond_m(S,S)+1)}{\Diamond_m(S,S)^2}} & \mbox{if $s(m-1)\neq s(m+1)$} \\ - \frac{1}{\Diamond_m(S,T)} \left(  \frac{\sqrt{P_{s(m)}(u)P_{t(m)}(u)}}{P_{s(m-1)}(u)}\right) & \mbox{otherwise} \end{array} \right. ,$$
and similarly for any $S,T$
$$(e_m)_{ST}=\left\{ \begin{array}{ll} \frac{\sqrt{P_{s(m)}(u)P_{t(m)}(u)}}{P_{s(m-1)}(u)} & \mbox{if $s(m-1)= s(m+1)$}\\
0 & \mbox{otherwise} \end{array} \right.$$
\end{thm}

We will give a geometric interpretation of the King polynomials $P_\lambda(u)$ and the Brauer diamonds $\Diamond_m(S,T)$ in the next two sections. This will allow us to define an action of the Brauer algebra $B_n(\delta)$ on the span of all $\delta$-restricted walks in $\Omega^n(\lambda)$.

\subsection{A geometric interpretation of the roots of the King polynomials}

Recall the definition of $\delta$-degree of singularity given in Definition \ref{regsing}.

\begin{thm}\label{king} Fix $\delta\in \mathbb{Z}$ and let $\lambda\in \Lambda$.
Let $m_\delta(\lambda)$ be the multiplicity of $\delta$ as a root of the King polynomial $P_\lambda(u)$. Then we have 
$$m_\delta(\lambda)\, = \, {\rm deg}_\delta(\lambda)-{\rm deg}_\delta(\emptyset).$$
In particular, we have that $P_\lambda(\delta)\neq 0$ if and only if $\lambda$ is $\delta$-regular.
\end{thm}

\begin{proof}
Write $a={\rm min}\{\#(\mbox{$\circ$'s in $x_\lambda$}), \#(\mbox{$\times$'s in $x_\lambda$})\}$ and let $\delta = 2m$ or $2m+1$ for some $m\in \mathbb{Z}$.
Using Example \ref{empty} and Lemma \ref{key} we have that for $m\geq 0$,
$$a=\#(\mbox{$\times$'s in $x_\lambda$}) = {\rm deg}_\delta(\lambda)={\rm deg}_\delta(\lambda) - {\rm deg}_\delta(\emptyset)$$
as ${\rm deg}_\delta(\emptyset)=0$; and for $m<0$ we have
\begin{eqnarray*}
a &=& \#(\mbox{$\circ$'s in $x_\lambda$})\\
&=& \#(\mbox{$\circ$'s in $x_\lambda$}) - \#(\mbox{$\times$'s in $x_\lambda$}) + \#(\mbox{$\times$'s in $x_\lambda$})\\
&=& m + \#(\mbox{$\times$'s in $x_\lambda$})\\
&=& \#(\mbox{$\times$'s in $x_\lambda$}) - \#(\mbox{$\times$'s in $x_\emptyset$})\\
&=& {\rm deg}_\delta(\lambda) - {\rm deg}_\delta(\emptyset).
\end{eqnarray*}
Thus it's enough to show that $m_\delta(\lambda) = a$.\\
Now by definition of $P_\lambda(u)$ and Proposition \ref{ylambda} we have that $m_\delta(\lambda)$ is precisely the number of  $\times$'s in $x_\lambda$ satisfying (\ref{admissibletimes}) added to the number of $\circ$'s in $x_\lambda$ satisfying (\ref{admissiblecirc}).
We can represent the sequence of $\times$'s and $\circ$'s appearing in $x_\lambda$ reading from left to right by a graph as follows. Start at $(0,0)$ and for each term in the sequence add $(1,0)$ if it is a $\circ$, or add $(0,1)$ if it is a $\times$.
The graph  is given in Figure \ref{graphpositivem} for $m\geq 0$ and in Figure \ref{graphnegativem} for $m<0$.

\begin{figure}[ht]
\includegraphics[width=8cm]{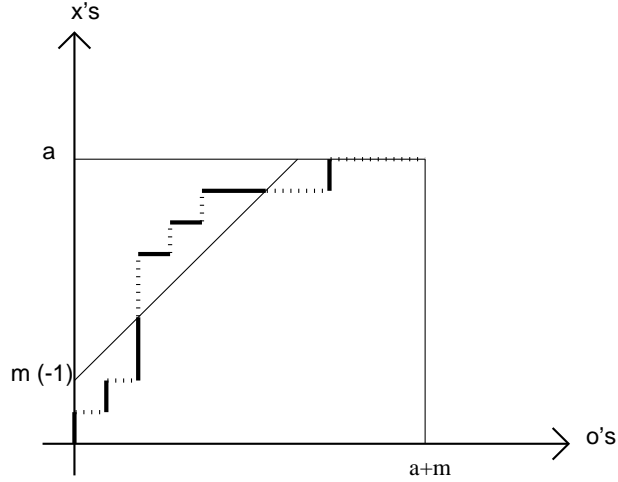}
\caption{Graph representing the sequence of $\times$ and $\circ$ in $x_\lambda$ for $m\geq 0$}
\label{graphpositivem}
\end{figure}

\begin{figure}[ht]
\includegraphics[width=8cm]{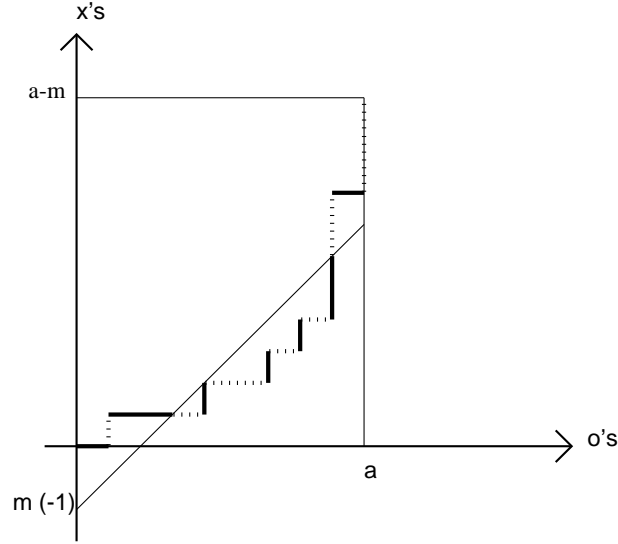}
\caption{Graph representing the sequence of $\times$ and $\circ$ in $x_\lambda$ for $m< 0$}
\label{graphnegativem}
\end{figure}

Now observe that the admissibility conditions (\ref{admissibletimes}) and (\ref{admissiblecirc}) can be rephrased as follows. A $\times$ (resp. $\circ$) satisfies (\ref{admissibletimes}) (resp. (\ref{admissiblecirc})) if and only if the corresponding step in the graph is below (resp. above) the line $y=x+m-\delta_{\delta,2m}$. Admissible (resp. non-admissible) steps are represented by solid lines (resp. dotted lines) in the graphs.
It follows immediately that the total number of admissible $\times$'s and $\circ$'s is equal to $a$.

\end{proof}

\begin{rem}
It was shown in \cite[Corollary (3.5)]{W} that $\lambda\in A_\delta$ if and only if $P_\mu(\delta)\neq 0$ for all $\mu\subseteq \lambda$. Theorem \ref{king} strengthens this result to give a full characterisation of the singularities of the King polynomial in terms of the $\delta$-degree of singularity.
\end{rem}

\subsection{A geometric interpretation of the Brauer diamonds}

In this section we give a geometric interpretation of the Brauer diamonds when we specialise $u=\delta$.

Recall the isomorphism between the Young graph $\mathcal{Y}$ and $\ZZZ_+(\rho_\delta)$ given in Section 2.2. Using this we will view walks on $\mathcal{Y}$ as walks in $\ZZZ_+(\rho_\delta)$ where each edge is of the form $x\rightarrow x\pm \epsilon_i$ for some $x\in \mathbb{R}^\mathbb{N}$ and some $i\geq 1$.

Let $(S,T)$ be an $m$-diamond pair. The Brauer diamond only depends on the $m-1$, $m$ and $m+1$ steps in the walks, so we will write $S=(x(m-1),x(m),x(m+1)) $ and $T=(x(m-1), y(m), x(m+1))$, where the $x(i)$'s and $y(i)$'s are in $\RR^\NN$. 

\begin{thm}\label{diamond} The Brauer diamonds satisfy the following identities.\\
\textbf{Case 1.} If $S=(x, x\pm \epsilon_i, x\pm \epsilon_i\pm \epsilon_j)$, $T=(x, x\pm \epsilon_j, x\pm \epsilon_i\pm \epsilon_j)$,\\
then we have for $i\neq j$
$$\Diamond(S,T)=\Diamond(T,S)=0$$
$$\Diamond(S,S)=-\Diamond(T,T)=\langle x,\epsilon_i - \epsilon_j\rangle ,$$
and for $i=j$ we have
$$\Diamond(S,S)=-1.$$
\textbf{Case 2.} If $S=(x, x\pm\epsilon_i, x\pm\epsilon_i \mp \epsilon_j)$, $T=(x, x\mp\epsilon_j, x\pm\epsilon_i \mp \epsilon_j)$ with $i\neq j$, \\
then we have
$$\Diamond(S,T)=\Diamond(T,S)=0$$
$$\Diamond(S,S)=-\Diamond(T,T)=\pm\langle x,\epsilon_i + \epsilon_j\rangle .$$ 
\textbf{Case 3.} If $S=(x, x+\alpha , x)$, $T=(x, x+\beta, x)$ with $\alpha, \beta\in \{\pm \epsilon_i\, : \, i\geq 1\}$,\\
then we have
$$\Diamond(S,T)=\Diamond (T,S)= \langle x, \alpha + \beta \rangle  +1$$
\end{thm}

\begin{proof}
We will give a proof for (half of) Case 1, the other cases can be computed similarly.
Going back to the original definition, we consider the Brauer diamond in $\mathcal{Y}$ given by $S=(\lambda, \lambda + \epsilon_k, \lambda + \epsilon_k + \epsilon_l)$, $T=(\lambda, \lambda + \epsilon_l, \lambda + \epsilon_k + \epsilon_l)$. 
Suppose that box $\epsilon_k$ is added in column $i$ and boxes $\epsilon_l$ is added in column $j$.
Note that for $S\neq T$ we must have $k\neq l$ and $i\neq j$. In this case we have
$$\Diamond(S,T)=(\lambda_l+1)-l-(\lambda_l+1)+l=0,$$
$$\Diamond(T,S)=(\lambda_k+1)-k-(\lambda_k+1)+k=0$$
and 
\begin{eqnarray*}
\Diamond(S,S)&=& (\lambda_l+1)-l-(\lambda_k+1)+k\\
&=& j-(\lambda'_j-1)-i+(\lambda'_i-1)\\
&=&(\lambda'_i -\frac{\delta}{2} -i+1)-(\lambda'_j - \frac{\delta}{2} -j+1)\\
&=&\langle e_\delta(\lambda), \epsilon_i - \epsilon_j\rangle\\
&=& -\Diamond(T,T).
\end{eqnarray*}
If $i=j$, then $l=k+1$ and we have
$$\Diamond(S,S)=(\lambda_{k+1}+1)-(k+1)-(\lambda_k+1)+k=-1.$$
Finally, if $k=l$ then $j=i+1$ and we have
\begin{eqnarray*}
\Diamond(S,S)&=& (\lambda_k +2)-k-(\lambda_k +1)+k\\
&=& 1 \\
&=& (\lambda'_i -\frac{\delta}{2}-i+1)-(\lambda'_{i+1}-\frac{\delta}{2}-(i+1)+1)\\
&=&\langle e_\delta(\lambda), \epsilon_i - \epsilon_{i+1}\rangle.
\end{eqnarray*}
\end{proof}

\begin{figure}[ht]
\includegraphics[width=10cm]{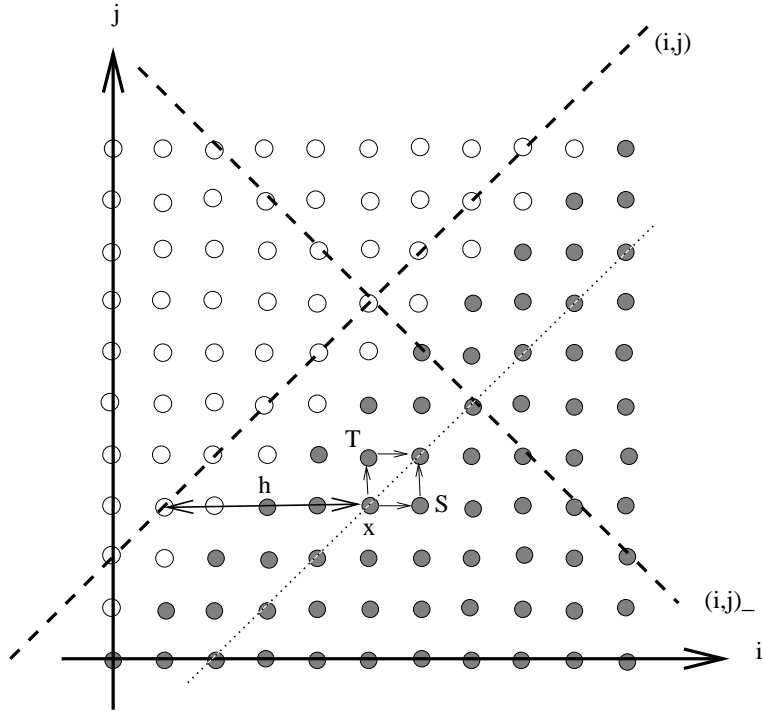}
\caption{\textbf{Case 1}: $S=(x,x+\epsilon_i, x+\epsilon_i+\epsilon_j)$, $T=(x,x+\epsilon_j, x+\epsilon_i+\epsilon_j)$ with $i<j$ and $h=\langle x, \epsilon_i-\epsilon_j\rangle$. 
Projection of $\mathbb{R}^{\mathbb{N}}$ onto the $ij$-plane, showing the reflection hyperplane. Fibres containing partitions are shaded.}
\label{ijplane1}
\end{figure}

\medskip

\begin{figure}[ht]
\includegraphics[width=10cm]{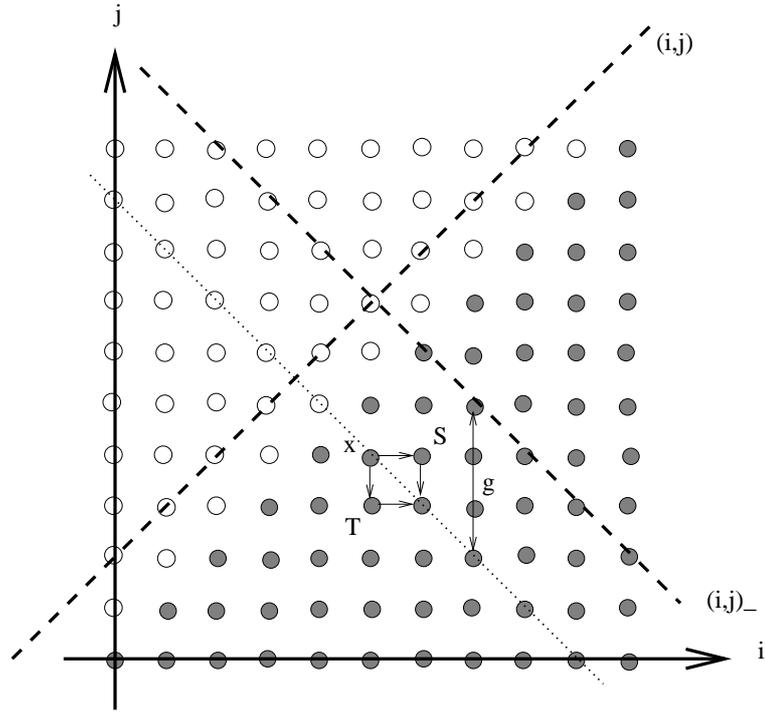}
\caption{\textbf{Case 2}: $S=(x,x+\epsilon_i, x+\epsilon_i-\epsilon_j)$, $T=(x,x-\epsilon_j, x+\epsilon_i-\epsilon_j)$ with $i<j$ and $g=\langle x,  \epsilon_i +\epsilon_j \rangle$}
\label{ijplane2}
\end{figure}

\medskip

\begin{figure}[ht]
\includegraphics[width=10cm]{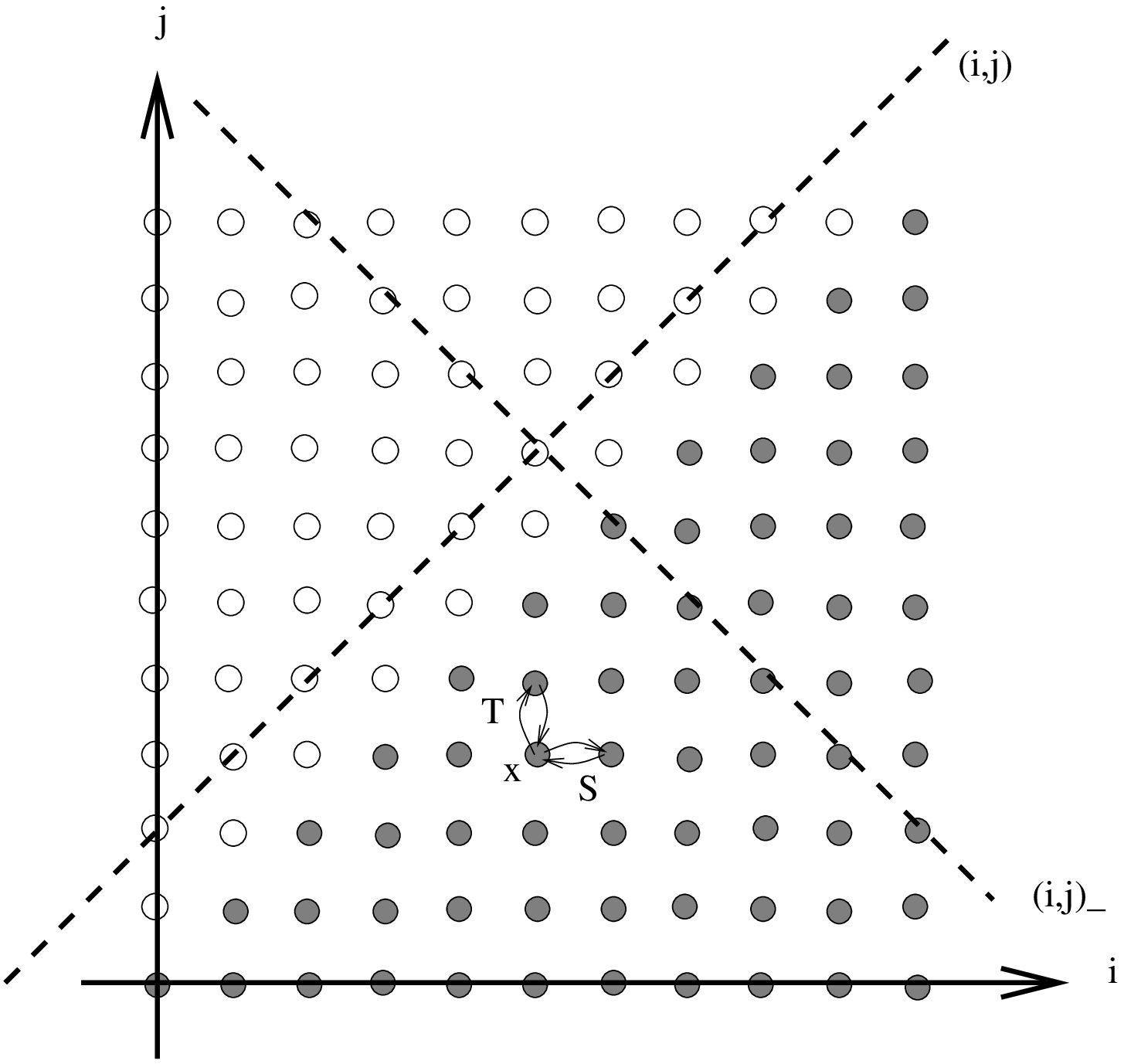}
\caption{\textbf{Case 3}: $S=(x,x+\epsilon_i, x)$, $T=(x,x+\epsilon_j, x)$ with $i<j$.}
\label{ijplane3}
\end{figure}

\medskip

\begin{figure}[ht]
\includegraphics[width=10cm]{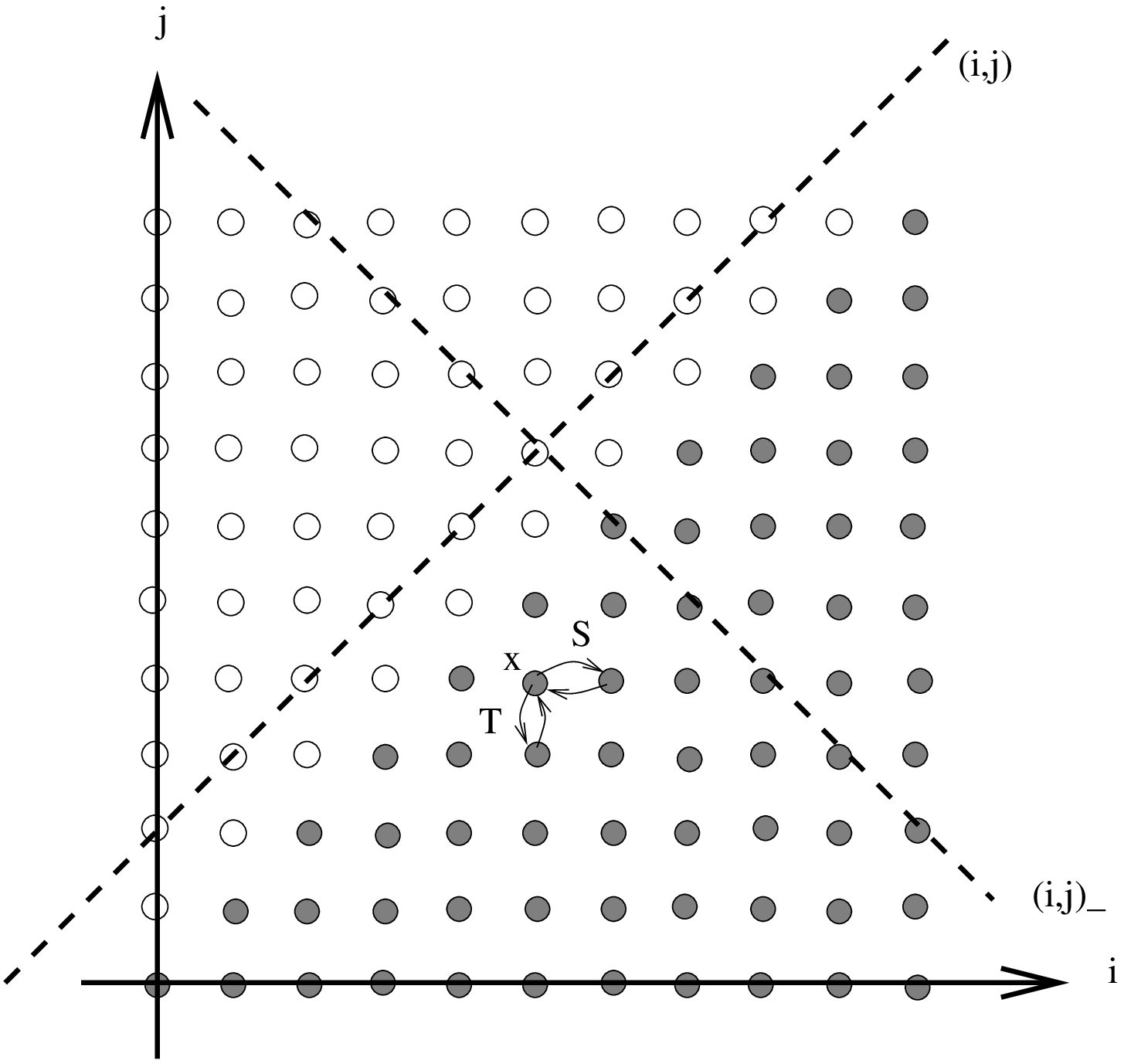}
\caption{\textbf{Case 3}: $S=(x,x+\epsilon_i, x)$, $T=(x,x-\epsilon_j, x)$ with $i<j$.}
\label{ijplane4}
\end{figure}

\subsection{Walk bases for $\delta$-restricted simple modules} 

Recall the definition of the set of $\delta$-restricted partitions $A_\delta$ and the $\delta$-restricted Young graph $\mathcal{Y}_\delta$ given in Section 2.2. We now show that if we truncate the Leduc-Ram representations $\Pi^\lambda$ to walks on $\mathcal{Y}_\delta$ then this gives a well-defined representation of $B_n(\delta)$, by specialising $u$ to $\delta$. We start by giving an explicit description of $A_\delta$.

\begin{prop}\label{A0} A partition $\lambda$ belongs to $A_\delta$ if and only if one of the following conditions holds.\\
(i) $\delta \geq 0$ and $\lambda^T_1 + \lambda^T_2 \leq \delta$.\\
(ii) $\delta =-2m$ (for some $m\in \NN$) and $\lambda_1\leq m$.\\
(iii) $\delta = -2m +1$ (for some $m\in \NN$) and $\lambda_1 + \lambda_2\leq 2m+1$.
\end{prop}

\begin{proof}
For $\delta = 2m$ or $2m+1$, the weight diagram $x_\emptyset$ consists of $m$ $\circ$'s (for $m\geq 0$) or $m$ $\times$ for $m<0$ followed by infinitely many $\vee$'s (see Figure \ref{emptypartition}).
Moreover, all possible configurations of translation equivalent weight diagrams are given in Lemma \ref{onebox} (i)-(v). It follows that the weight diagrams corresponding to partitions in $A_\delta$ are precisely those having $m$ $\circ$ (for $m\geq 0$) or $m$ $\times$ (for $m<0$) and with the other vertices either all labelled by $\vee$'s, or labelled by one $\wedge$ and infinitely many $\vee$'s, in that order. The result then follows from the end of Section 2.2 (see Figure \ref{weighttranspose} and \ref{weightpartition}).
\end{proof}

\begin{rem}
Proposition \ref{A0} also follows by combining \cite{W} (definition before Theorem (3.4) and Corollary (3.5)(b)) with Theorem \ref{king}.
\end{rem}

\begin{thm}\cite[Theorem 2.4(b)]{RamWenzl92}
Let $\lambda\in A_\delta$. Then there is an action of $B_n(\delta)$ on the vector space spanned by all walks on $\mathcal{Y}_\delta$ from $\emptyset$ to $\lambda$. This module is isomorphic to $L_n(\lambda)$.
\end{thm}

\begin{proof}
We consider the action of the generic Brauer algebra $B_n(u)$ on the Leduc-Ram representations $\Pi^{\lambda}$ and claim that the truncation of this action to $\delta$-restricted walks gives a well-defined representation by setting $u=\delta$. This requires two things:\\
(1) that all matrix entries $(\sigma_m)_{ST}, (e_m)_{ST}$ where $S,T$ are $\delta$-restricted walks are well defined,\\
(2) that the matrix entries $(\sigma_m)_{ST}, (e_m)_{ST}$ where $T$ is $\delta$-restricted but $S$ is not are all zero.

Note that if $(\sigma_m)_{ST}$ or $(e_m)_{ST}$ are non-zero, then $(S,T)$ is a Brauer diamond. So we will consider the three cases of Brauer diamonds given in Theorem \ref{diamond}.

\textbf{Case 1.} $S=(x, x\pm \epsilon_i, x\pm \epsilon_i\pm\epsilon_j)$, $T=(x, x\pm\epsilon_j, x\pm \epsilon_i\pm\epsilon_j)$.\\
For note that the submatrix of $(e_m)$ mixing between $S$ and $T$ is identically zero as $s(m-1)\neq s(m+1)$. So there is nothing to check here.\\
For $i=j$ we have $S=T$ and $(\sigma_m)_{SS}=-1$.
For $i\neq j$, write $h=\langle x, \epsilon_i - \epsilon_j\rangle$. As $x$ is strictly decreasing we have that $h\neq 0$. So, the submatrix of the matrix $(\sigma_m)$ mixing between the walks $S$ and $T$ given by
$$\left(\begin{array}{cc} \frac{1}{h} & \frac{\sqrt{h^2-1}}{|h|} \\ \frac{\sqrt{h^2-1}}{|h|} & -\frac{1}{h}\end{array} \right)$$
is always well-defined, see Figure \ref{ijplane1}. This proves (1). \\
Observe that of $T$ is $\delta$-regular, then so is $S$. Indeed, if $S$ was $\delta$-singular, then $\langle x\pm\epsilon_i , \epsilon_i +\epsilon_k\rangle$ would be zero for some $k$. But as $x\pm \epsilon_i\pm\epsilon_j$ is $\delta$-regular, we must have $k=j$. But then we would have $\langle x\pm\epsilon_j , \epsilon_i +\epsilon_j\rangle = \langle x\pm\epsilon_i , \epsilon_i +\epsilon_k\rangle = 0$ which contradicts the fact that $T$ is $\delta$-regular. So there is noting to check for (2) in this case.

\medskip

\textbf{Case 2.} $S=(x, x\pm \epsilon_i, x\pm \epsilon_i \mp \epsilon_j)$, $T=(x, x\mp \epsilon_j, x\pm\epsilon_i \mp \epsilon_j)$ with $i\neq j$.\\
As in Case 1, we have that the submatrix of $(e_m)$ is identically zero in this case. Write $g=\langle x, \epsilon_i + \epsilon_j\rangle$. Then the submatrix of $(\sigma_m)$ mixing between the walks $S$ and $T$ is given by
$$\left(\begin{array}{cc} \pm \frac{1}{g} & \frac{\sqrt{g^2-1}}{|g|} \\ \frac{\sqrt{g^2-1}}{|g|} & \mp \frac{1}{g}\end{array} \right),$$
(see Figure \ref{ijplane2}).
Now we claim that if $T$ is $\delta$-regular, then we cannot have $g=0$. Indeed, for $T$ $\delta$-regular, we have that $x$, $x\mp \epsilon_j$ and $x\mp \epsilon_j \pm \epsilon_i$ all have the same degree of singularity. Now if $g=\langle x, \epsilon_i+\epsilon_j\rangle =0$ then we have $x_j=-x_i$. But then, $(x\mp \epsilon_j)_j=-x_i\mp 1$ and as $x\mp \epsilon_j$ has the same degree of singularity as $x$ we must have that $x\mp \epsilon_j$ has a coordinate equal to $x_i\pm 1$. Thus $x\mp \epsilon_j$ has both entries $x_i$ and $x_i\pm 1$. But this would imply that $x\mp \epsilon_j \pm \epsilon_i$ is not strictly decreasing, which is a contradiction. Hence we have shown that $g$ cannot be zero and the matrix entries are all well-defined, proving (1).

Now suppose that $T$ is $\delta$-regular but $S$ is not. So we have that $x$ is $\delta$-regular and $x\pm \epsilon_i$ is not. Thus we have that $x_i\pm 1=-x_h$ for some $h$. But as $x\pm \epsilon_i\mp \epsilon_j$ is $\delta$-regular, we have that $h=j$ and so $x_j=-x_i\mp 1$. This shows that $g= \langle x, \epsilon_i + \epsilon_j\rangle =\mp 1$ and hence $g^2-1=0$. This proves (2).

\medskip

\textbf{Case 3.}  $S=(x, x+ \alpha, x)$, $T=(x, x+ \beta, x)$ where $\alpha, \beta\in \{\pm \epsilon_i \, : \, i\geq 1\}$, see Figures \ref{ijplane3} and \ref{ijplane4}.
In this case the submatrix of $(\sigma_m)$ mixing between $S$ and $T$ is given by
$$\left(\begin{array}{cc} \frac{1}{2\langle x, \alpha\rangle +1}\left( 1-\frac{P_{x+\alpha}(\delta)}{P_{x}(\delta)}\right) & \frac{-1}{ \langle x, \alpha + \beta \rangle +1}\left(\frac{\sqrt{P_{x+\alpha}(\delta)P_{x+\beta}(\delta)}}{P_x(\delta)}\right) \\
\frac{-1}{\langle x, \alpha + \beta \rangle +1}\left(\frac{\sqrt{P_{x+\alpha }(\delta)P_{x+\beta}(\delta)}}{P_x(\delta)}\right) & \frac{1}{2\langle x, \beta \rangle +1}\left( 1-\frac{P_{x+\beta}(\delta)}{P_{x}(\delta)}\right)\end{array}\right)$$
and the submatrix of $(e_m)$ mixing between $S$ and $T$ is given by
$$\left(\begin{array}{cc} \frac{|P_{x+\alpha}(\delta)|}{P_{x+\alpha}(\delta)} & \frac{\sqrt{P_{x+\alpha }(\delta)P_{x+\beta}(\delta)}}{P_x(\delta)} \\
\frac{\sqrt{P_{x+\alpha}(\delta)P_{x+\beta}(\delta)}}{P_x(\delta)} & \frac{|P_{x+\beta}(\delta)|}{P_{x+\beta}(\delta)}\end{array}\right).$$
First note that if $T$ is $\delta$-regular, then using Theorem 7.2 we have that $P_x(\delta)\neq 0$. Thus the entries in the submatrix representing the action of $e_m$ are all well-defined. 

Now suppose that we had $\langle x, \alpha + \beta\rangle +1 =0$, that is $\langle x, \alpha + \beta \rangle = -1$. So we get $\langle x+\beta , \alpha + \beta\rangle =0$. Now $\alpha + \beta  = \pm (\epsilon_i\pm \epsilon_j$ for some $i,j$ and we can assume $i\neq j$ as $S\neq T$ and $\alpha +\beta \neq 0$. Moreover, as $x+\beta$ is strictly decreasing we cannot have $\alpha + \beta =\pm (\epsilon_i - \epsilon_j)$. Now suppose $\langle x+\beta , \epsilon_i+\epsilon_j\rangle =0$, with $\alpha = \pm \epsilon_i$ and $\beta =\pm \epsilon_j$. As $T$ is $\delta$-regular, we have that $x$ and $x+\beta=x\pm \epsilon_j$ have the same degree of singularity. Thus $x$ must have an entry equal to $x_i \pm 1$ (in position $i\pm 1$). But then $x+\alpha=x\pm \epsilon_i$ is not strictly decreasing, which is a contradiction. This proves that the off-diagonal entries of $(\sigma_m)$ are well-defined.

Now, if $T$ is $\delta$-regular but $S$ is not then we have that the off diagonal entries in $(\sigma_m)$ and $(e_m)$ are all zero using Theorem 7.2. This proves (2).

Now we claim that the diagonal entries in $(\sigma_m)$  are also well-defined.
Observe that it is possible to have $2\langle x , \alpha \rangle +1 =0$. However we claim that, as a polynomial in $\delta$, $P_x(\delta)-P_{x+\alpha}(\delta)$ is divisible by $2\langle x, \alpha \rangle +1$. To see this, note that before specialisation, the matrix $(\sigma_m (u))$ gives a well-defined representation of $B_n(u)$ and so we have $(\sigma_m(u))^2=I$ the identity matrix.
In particular we have that
$$\sum_{T\in \Omega_m(S)}(\sigma_m(u))_{ST}(\sigma_m(u))_{TS}=1.$$
So we have
$$(\sigma_m(u))_{SS}^2 + \sum_{\begin{subarray}{c} T\in \Omega_m(S) \\ T\neq S\end{subarray}}(\sigma_m(u))_{ST}^2=1.$$
Now we have seen that $\lim_{u\rightarrow \delta} (\sigma_m(u))_{ST}$ for all $T\in \Omega_m(S)$ and $T\neq S$ exist and are finite. Thus we must have that $\lim_{u\rightarrow \delta}(\sigma_m(u))^2_{SS}$ exists and is finite. This means that $(\sigma_m(u))_{SS}$ is a rational function with no poles at $u=\delta$, proving our claim. This completes the proof of (1).

\medskip

It remains to show that this module is isomorphic to $L_n(\lambda)$, defined as the simple head of the standard module $\Delta_n(\lambda)$. Denote the representation of $B_n(\delta)$ on $\delta$-restrited walks defined above by $\tilde{L}_n(\lambda)$. By looking at the action of the generators $\sigma_m$ and $e_m$ , we immediately see that
$$\res_n \tilde{L}_n(\lambda)\cong \bigoplus_{\lambda' \in \supp(\lambda)\cap A_\delta}\tilde{L}_{n-1}(\lambda')$$
We will prove by induction on $n$ that $\tilde{L}_n(\lambda)\cong L_n(\lambda)$. If $n=0$ then there is nothing to prove. Assume that the result holds for $n-1$. 
Let $\lambda'\in \supp(\lambda)\cap A_\delta$ (note that as $\delta\neq 0$, we have $\supp(\lambda)\cap A_\delta \neq 0$). Then we have
\begin{eqnarray*}
\Hom_n(\Delta_n(\lambda) , \tilde{L}_n(\lambda)) &\cong & \Hom_n(\ind_{n-1}^\lambda \Delta_{n-1}(\lambda'), \tilde{L}_n(\lambda))\\
&\cong & \Hom_{n-1}(\Delta_{n-1}(\lambda'), \res_n^{\lambda'} \tilde{L}_n(\lambda))\\
&\cong& \Hom_{n-1}(\Delta_{n-1}(\lambda'), {\rm pr}^{\lambda'} \oplus_{\mu\in \supp(\lambda)\cap A_\delta}\tilde{L}_{n-1}(\mu))\\
&\cong&  \Hom_{n-1}(\Delta_{n-1}(\lambda'), {\rm pr}^{\lambda'} \oplus_{\mu\in \supp(\lambda)\cap A_\delta}L_{n-1}(\mu)) \quad \mbox{by induction}\\
&\cong & \Hom_{n-1}(\Delta_{n-1}(\lambda'), L_{n-1}(\lambda'))\\
&=& \CC.
\end{eqnarray*}
This shows that $\tilde{L}_n(\lambda)$ contains $L_n(\lambda)$ as a composition factor. But using Corollary \ref{dimsimple}, we have that $\dim L_n(\lambda)=\dim \tilde{L}_n(\lambda)$ and so we must have $\tilde{L}_n(\lambda) \cong L_n(\lambda)$.

\end{proof}

\end{document}